\documentclass[reqno, 12pt]{amsart}
\usepackage[margin=1.25in]{geometry}
\usepackage{amsthm, amsmath, amsfonts, amssymb}
\usepackage{enumerate, mathabx}

\usepackage{color}

\newcommand{\Z}[0]{\mathbb{Z}}
\newcommand{\R}[0]{\mathbb{R}}
\newcommand{\Q}[0]{\mathbb{Q}}
\newcommand{\C}[0]{\mathbb{C}}
\newcommand{\N}[0]{\mathbb{N}}

\newcommand{\ta}[0]{\theta}
\newcommand{\bs}[0]{\setminus}
\newcommand{\ld}[0]{\lambda}
\newcommand{\vep}[0]{\varepsilon}
\newcommand{\supp}[0]{\operatorname{supp}}
\newcommand{\lsm}[0]{\lesssim}

\newcommand{\ow}[0]{\overline{w}}
\newcommand{\avg}[1]{\underset{#1}{\operatorname{Avg}}}

\newcommand{\mb}[1]{\mathbf{#1}}
\newcommand{\mf}[1]{\mathfrak{#1}}
\newcommand{\wh}[1]{\widehat{#1}}

\newcommand{\mc}[1]{\mathcal{#1}}
\newcommand{\ov}[1]{\overline{#1}}
\newcommand{\wt}[1]{\widetilde{#1}}
\newcommand{\st}[1]{\substack{#1}}

\newcommand{\nms}[1]{\| #1 \|}

\newcommand{\geom}[0]{\operatorname{geom}}

\newcommand{\E}[0]{\mathcal{E}}

\newtheorem{thm}{Theorem}[section]
\newtheorem{lemma}[thm]{Lemma}
\newtheorem{prop}[thm]{Proposition}
\newtheorem{cor}[thm]{Corollary}

\theoremstyle{remark}
\newtheorem{rem}[thm]{Remark}

\title{Effective $l^2$ decoupling for the parabola}
\author[Zane Kun Li]{Zane Kun Li\\(with an appendix by Jean Bourgain and Zane Kun Li)}
\address{Department of Mathematics,
Indiana University\\
831 East 3rd St\\
Bloomington IN 47405\\
USA}
\email{zkli@iu.edu}

\begin{document}
\begin{abstract}
We make effective $l^2 L^p$ decoupling for the parabola in the range $4 < p < 6$.
In an appendix joint with Jean Bourgain, we apply the main theorem to prove
the conjectural bound for the sixth-order correlation of the integer solutions
of the equation $x^2 + y^2 = m$ in an extremal case.
This proves unconditionally a result that was proven in \cite{bombieribourgain}
under the hypotheses of the Birch and Swinnerton-Dyer conjecture and the Riemann Hypothesis for $L$-functions of
elliptic curves over $\mathbb{Q}$.
\end{abstract}

\maketitle
\section{Introduction}\label{intro}
For each interval $J \subset [0, 1]$ and $g: [0, 1] \rightarrow \C$, let
\begin{align*}
(\E_{J}g)(x) := \int_{J}g(\xi)e(\xi x_1 + \xi^{2}x_2)\, d\xi
\end{align*}
where $e(z) = e^{2\pi i z}$.
For a square $B = B(c, R) \subset \R^2$ centered at $c$ of side length $R$,
let $$w_{B}(x) := (1 + \frac{|x - c|}{R})^{-100}.$$
If $I$ is an interval in $[0, 1]$ and $\delta \in (0, 1)$, let $P_{\delta}(I)$
be the partition of $I$ into $|I|/\delta$ many intervals of length $\delta$.
Note that when writing $P_{\delta}(I)$, we assume $|I|/\delta \in \N$.
For $\delta \in \N^{-1}$ and $p \geq 1$, let $D_{p}(\delta)$ be the best constant such that
\begin{align}\label{decdelta}
\nms{\E_{[0, 1]}g}_{L^{p}(B)} \leq D_{p}(\delta)(\sum_{J \in P_{\delta}([0, 1])}\nms{\E_{J}g}_{L^{p}(w_B)}^{2})^{1/2}
\end{align}
for all $g: [0, 1] \rightarrow \C$ and all (axis-parallel) squares $B \subset \R^2$ of side length $\delta^{-2}$.
Observe that trivially we have that $D_{p}(\delta)\leq 2^{100/p}\delta^{-1/2}$ and $D_{2}(\delta) \lsm 1$ from Plancherel.

Bourgain and Demeter showed in \cite{bd} that $D_{p}(\delta) \lsm_{\vep, p} \delta^{-\vep}$ for $2 \leq p \leq 6$
and this range of $p$ is sharp up to a $\delta^{-\vep}$-loss (in fact they showed that the decoupling constant for the paraboloid in $\R^d$ is
$O_{p,\vep}(\delta^{-\vep})$ for $2 \leq p \leq \frac{2(d + 1)}{d - 1}$).
The main result of this paper is to obtain the first known improvements beyond a bound of $O_{\vep, p}(\delta^{-\vep})$ in the range $4 < p < 6$.
\begin{thm}\label{main}
If $4 < p < 6$ and $\delta \in \N^{-1}$, then for some sufficiently large absolute constant $C$ we have
\begin{align*}
D_{p}(\delta) \leq \exp(C(\log\frac{1}{\delta})^{\frac{3}{4} + \frac{1}{4}\log_{2}(\frac{p - 2}{2})}\log\log\frac{1}{\delta}).
\end{align*}
\end{thm}
A version of Theorem \ref{main} with explicit constants was obtained by the author in his thesis \cite[Chapter 2]{thesis}.
There the author obtained a quantitatively better but qualitatively similar result, however
the argument presented here is somewhat simpler and more straightforward in the iterative step.

We put Theorem \ref{main} in some context.
It is conjectured that $D_{p}(\delta) \lsm C_{p}$ for all $2 \leq p < 6$ (see \cite{bd} and the same conjecture can be made for the decoupling
constant for the paraboloid in $\R^d$ and $2 \leq p < \frac{2(d + 1)}{d - 1}$).
By picking $g = 1_{[0, \delta]}$ we see that $D_{p}(\delta) \gtrsim 1$ for any $p \geq 1$.
Bounds for $D_{p}(\delta)$ in the range $2 \leq p \leq 4$ can be obtained from first directly proving that
$D_{4}(\delta) \lsm 1$ (see \cite[Section 4.3]{thesis}) and interpolating with $D_{2}(\delta) \lsm 1$ to obtain
that $D_{p}(\delta) \lsm C_{p}$ in the range $2 \leq p \leq 4$ (see also \cite[Remark 3.4]{jlee} for how to interpolate
decoupling estimates).
At the endpoint $p = 6$, the current best known estimates are
\begin{align}\label{endpointbd}
(\log\frac{1}{\delta})^{1/6} \lsm D_{6}(\delta) \leq \exp(O(\frac{\log\frac{1}{\delta}}{\log\log\frac{1}{\delta}}))
\end{align}
where the first inequality was obtained by Bourgain in \cite[Eq. (2.51)]{gafa93} and the second inequality was obtained by
the author in \cite[Theorem 1.1]{ef2d} (a related estimate was obtained by Bourgain using the divisor bound in the discrete Fourier restriction
case, see \cite[Proposition 2.36]{gafa93}).
It is unknown whether the upper or lower bounds in \eqref{endpointbd} can be improved.
See the bottom of Page 3 of \cite{wooleyicm} for a similar phenomenon in the (classical) $s = 3, k = 2$ case of Vinogradov's mean value theorem
though in this case we have precise asymptotics \cite{blomerbrudern}.

The initial motivation for the author to obtain effective estimates for the parabola decoupling constant was
to obtain effective estimates for the decoupling constant for the moment curve $t \mapsto (t, t^2, \ldots, t^k)$.
One key ingredient of all current proofs of decoupling for the moment curve (\cite{bdg} and \cite{glyzk}) is that decoupling
for the moment curve $t \mapsto (t, t^2, \ldots, t^k)$ requires knowledge of decoupling for the moment curves in all lower dimensions, that is
decoupling for the curves $t \mapsto (t, t^2)$, $t \mapsto (t, t^2, t^3)$, and so until $t \mapsto (t, t^2, \ldots, t^{k - 1})$.
One strategy to obtain an effective estimate for the decoupling constant for the moment curve $t \mapsto (t, t^2, \ldots, t^k)$
is then to first obtain a good effective estimate for parabola decoupling. Once we have this estimate for the parabola
we can use it to obtain an effective estimate for decoupling for the curve
$t \mapsto (t, t^2, t^3)$ which would then lead to an effective estimate for decoupling for the curve $t \mapsto (t, t^2, t^3, t^4)$, and so on.

The dependence on $k$ and $\vep$ in the moment curve decoupling constant is important to improve current known estimates on the zeta
function (see \cite{bourgain, ford1, ford2, heathbrown} and the MathOverflow question \cite{mathoverflow}).
Unfortunately because of the iterative nature of all proofs of Vinogradov's mean value theorem
\cite{bdg}, \cite{nested} (and the short proof \cite{glyzk} which is inspired from \cite{nested}), the dependence on $k$ is currently quantitatively rather poor and so yields
little progress beyond improving the implied constants in current known estimates for the zeta function.

Having obtained an effective bound for decoupling for the parabola in Theorem \ref{main}, in an appendix with
Jean Bourgain we consider the following problem on the sixth-order correlation for integer lattice points
on the circle $x^{2} + y^{2} = m$.
Let $\Lambda_m$ be the set of Gaussian integers $\ld$ with norm $m$ and $N = |\Lambda_m|$. Next, let $S_{6}(m)$
be the number of six tuples $(\ld_1, \ld_2, \ldots, \ld_6) \in  \Lambda_{m}^6$ such that $$\ld_1 + \ld_2 + \ld_3 = \ld_4 + \ld_5 + \ld_6.$$
Trivially we have $S_{6}(m) = O(N^4)$ and a bound of $O_{\vep}(N^{3+ \vep})$ is conjectured.
Bourgain in \cite[Theorem 2.2]{kkw} showed that $S_{6}(m) = o(N^4)$ as $N \rightarrow \infty$.
In \cite[Section 2]{bombieribourgain}, he showed that $S_{6}(m) = O(N^{7/2})$, making use of the Szemer\'edi-Trotter theorem.
As an application of Theorem \ref{main}, we prove the following theorem.
\begin{thm}\label{main2}
Let $A = \{(x, y) \in \Z^2: x^2 + y^2 = R\}$ and $R$ a sufficiently large integer. Then
\begin{align}\label{m2eq1}
|\{(z_1, \ldots, z_6) \in A^6: z_1 + z_2 + z_3 = z_4 + z_5 + z_6\}| \ll |A|^{3 + o(1)}
\end{align}
provided $|A| > \exp((\log R)^{1 - o(1)})$.
\end{thm}
Theorem \ref{main2} then gives an unconditional proof of \cite[Theorem 25]{bombieribourgain} (see also (126) in that paper) which was proved
in only a random sense and conditionally on the Birch and Swinnerton-Dyer conjecture and the Riemann Hypothesis for $L$-functions
of elliptic curves over $\Q$. Note that from the divisor bound on $\Z[i]$, $|A| \leq \exp(O(\frac{\log R}{\log\log R}))$ with equality attained for infinitely many $R$
and so the hypothesis of Theorem \ref{main2} covers the extremal case of $A$.

Theorem \ref{main2} is a consequence of the following square root cancellation estimate which we state (somewhat) loosely as:
Let $A$ be a set of points $(x, y)$ with $x^{2} + y^{2} = 1$ and separated by distance at least $\delta$.
Then for $4 < p < 6$,
\begin{align}\label{targetexpsum}
\nms{\sum_{a \in A}e^{2\pi i a \cdot z}}_{L^{p}_{\#}(B)} \leq \exp(O((\log\frac{1}{\delta})^{\frac{3}{4} + \frac{1}{4}\log_{2}(\frac{p - 2}{2})}(\log\log\frac{1}{\delta})^{O(1)}))|A|^{1/2}
\end{align}
for all squares $B$ of side length $O_{\vep}(\delta^{-3 - \vep})$ (see Proposition \ref{srtc} for a more precise statement).

We remark that if the set $A$ in Theorem \ref{main2} was replaced instead with the set $\{(x, y) \in \Z^2: 1 \leq x \leq R, y = x^2\}$,
then \eqref{m2eq1} is just Vinogradov's mean value theorem in the case when $s = 3, k = 2$. Thus Theorem \ref{main2} could be interpreted
as a sort of Vinogradov's mean value theorem where we replace the underlying parabola with a circle.
A key step is to establish a version of Theorem \ref{main} where instead of considering decoupling
for the parabola, we consider decoupling for the circle. We will also need to make sure that the implied constant
in such a result is uniform in $p$.

Since the $s = 3, k = 2$ case of Vinogradov's mean value theorem is an immediate consequence
of the estimate $D_{6}(\delta) \lsm_{\vep} \delta^{-\vep}$, one can ask perhaps why we can't use the analogue of this estimate for the circle
and why we are using the estimate for $D_{p}(\delta)$ in the range $4 < p < 6$ rather than the current best
known upper bound at the endpoint from \eqref{endpointbd}.

First, it was sketched in \cite[Section 7]{bd} that the analogue of $D_{6}(\delta) \lsm_{\vep} \delta^{-\vep}$ for the circle holds.
Applying this would yield \eqref{m2eq1} with the right hand side replaced
with $\lsm_{\vep} R^{\vep}|A|^{3}$. However, apriori the $R^{\vep}$ could be much larger than $|A|$ and so it is desirable to obtain the explicit dependence on $\vep$
to give a small improvement over $R^{\vep}$.

Second, our strategy in proving \eqref{targetexpsum} is to upgrade an estimate about decoupling for the parabola
to an estimate about decoupling for the circle. To do this we use the Pramanik-Seeger iteration \cite{pramanikseeger} and
keep careful track of the dependence on various implied constants.
Since the Pramanik-Seeger iteration requires another $\log\log \frac{1}{\delta}$ many iterations, upgrading the second inequality in \eqref{endpointbd} to the case
of the circle would give a bound that is worse than trivial (see Lemma \ref{dbd} and Corollary \ref{dbdcor} and compare with Theorem \ref{main}) since the trivial bound is just $\delta^{-O(1)}$.
Thus to get around these technical issues, we work instead with Theorem \ref{main} to show \eqref{targetexpsum} and then (roughly) take a limit as $p \rightarrow 6$
which allows us to recover the desired square-root cancellation estimate Lemma \ref{mainlemma} and in turn leads us to Theorem \ref{main2}.

\subsection*{Organization of the paper}
In Section \ref{weightsec}, we introduce all the properties of the weight functions $w_B, \ow_B$, and $\eta_B$ which we
will use. In Section \ref{basic} we mention basic properties of the decoupling constant such as
parabolic rescaling, a bilinear constant and bilinearization. Next in Section \ref{ball} we prove a ball inflation lemma
and apply this key lemma to the iteration in Section \ref{iteration} which proves Theorem \ref{main}.
Finally we end with an appendix joint with Jean Bourgain where we prove Theorem \ref{main2}.

While the proof of Theorem \ref{main} largely follows the outline of Bourgain and Demeter in \cite{sg} and works out the
explicit dependence on $p$, some significant differences between this work and the work of Bourgain and Demeter are as follows:
\begin{itemize}
\item We have two weight functions $w_B$ and $\ow_B$ (defined at the beginning of Section \ref{weightsec}) rather than just one; the weight function $\ow_B$ allows for
easier computation of weight functions in proofs that involving any sort of rescaling since we are now allowed to weight each direction differently.
The weight function $\ow_B$ also simplifies and clarifies the proof of parabolic rescaling in Lemma \ref{parabolic_rescaling}.
\item The bilinear constant defined in \eqref{mb_bds} is different from the bilinear constant presented in \cite[Eq. (14)]{sg}.
Both the iteration and bilinear constant are much closer to those presented in \cite{ef2d} rather in \cite{sg}.
Because we are just working in $\R^2$, the reduction to the bilinear decoupling constant in
Lemma \ref{reduction} just uses H\"{o}lder rather than a Bourgain-Guth argument.
To also obtain a quantitative estimate we make explicit a recursive relation about $D(\delta)$ in
Corollary \ref{rcor}.
\item While the statement of the ball inflation step in Lemma \ref{ballinflation} is similar to \cite[Theorem 9.2]{sg},
since we are just working in $\R^2$, the proof simplifies since the relevant geometric input is about the intersection of
transverse rectangles in $\R^2$. See the paragraph after \eqref{kakeyapar}.
\end{itemize}

\subsection*{Notation}
When we write $A \lsm B$, we mean that there is an absolute constant $C$ (independent of $p$, $\delta$, $\nu$)
such that $A \leq CB$.
If we write $A \lsm_{\vep} B$, then we mean $A \leq C_{\vep}B$ for some constant $C_{\vep}$ depending only on $\vep$.
Unless otherwise stated $p \in (4, 6)$. This will remove some dependence
on $p$ that would otherwise have been present.

We will write $P_{\delta}(I)$ to be the partition of an interval $I$ into intervals of size $\delta$
and we similarly we define $P_{\nu^{-b}}(B)$ to be the partition of the square $B$ into squares of side length $\nu^{-b}$.
Note that this implicitly uses that $|I|/\delta$ and $\ell(B)/\nu^{-b} \in \N$ where $\ell(B)$ is the side length of $B$.

We also define
\begin{align*}
\nms{f}_{L^{p}(w_B)} &:= (\int |f|^{p}w_B)^{1/p},\\
\nms{f}_{L^{p}_{\#}(B)} &:= (\frac{1}{|B|}\int_{B}|f|^{p})^{1/p},
\end{align*}
and
$$\nms{f}_{L^{p}_{\#}(w_B)} := (\frac{1}{|B|}\int |f|^{p}w_B)^{1/p}.$$
Finally for two expressions $x_1$ and $x_2$ we let $\geom x_i := (x_1 x_2)^{1/2}$.

\subsection*{Acknowledgements}
The author would like to thank Ciprian Demeter, Larry Guth, and his thesis advisor Terence Tao for encouragement and many fruitful discussions
on decoupling.
The author is extremely grateful to Jean Bourgain for kindly pointing out the application of Theorem \ref{main} to proving an unconditional
version of Theorem 25 in \cite{bombieribourgain} and for allowing his argument to become an appendix to this paper.
The author would also like to thank the anonymous referee for useful comments and suggestions.
This work was completed while the author was supported by NSF grants DGE-1144087, DMS-1266164, and DMS-1902763.

\section{Properties of weight functions}\label{weightsec}
\subsection{The weights $w_B$ and $\ow_B$}
For a square $B = B(c, R)$ with center $c \in \R^2$ and side length $R$, define
$$w_{B}(x) := (1 + \frac{|x - c_B|}{R})^{-100}$$
and $$\ow_{B}(x) := (1 + \frac{|x_1 - c_{B1}|}{R})^{-100}(1 + \frac{|x_2 - c_{B2}|}{R})^{-100}.$$
Note that $\ow_B \leq w_B \leq \ow_{B}^{1/2}$.

We let $\eta$ be a Schwartz function
such that $\eta \geq 1_{B(0, 1)}$ and $\supp(\wh{\eta}) \subset B(0, 1)$.
If $B = B(c, R)$, we define $\eta_{B}(x) := \eta(\frac{x - c}{R})$.
Note that $\eta_{B} \lsm \ow_{B}$ and $\eta_{B} \lsm w_B$. An explicit
example of such an $\eta$ was constructed in \cite[Corollary 2.2.9]{thesis}.

The following lemma is the key property satisfied by $w_B$ and $\ow_B$.
\begin{lemma}\label{wbconvolve}
For $0 < R' \leq R$,
\begin{align}\label{wbconveq1}
w_{B(0, R)} \ast w_{B(0, R')} \lsm R'^{2}w_{B(0, R)}.
\end{align}
We also have
\begin{align}\label{wbconveq2}
R^{2}w_{B(0, R)} \lsm 1_{B(0, R)} \ast w_{B(0, R)}.
\end{align}
The same inequalities hold true if $w_{B(0, R)}$ is replaced with $\ow_{B(0, R)}$.
\end{lemma}
\begin{proof}
We first prove \eqref{wbconveq1}.
We would like to give an upper bound for the expression
\begin{align*}
\frac{1}{R'^{2}}\int_{\R^2}(1 + \frac{|x - y|}{R})^{-100}(1 + \frac{|y|}{R'})^{-100}(1 + \frac{|x|}{R})^{100}\, dy.
\end{align*}
A change of variables in $y$ and rescaling $x$ shows that it suffices to give an upper bound for
\begin{align}\label{coeq3}
\int_{\R^2}(1 + |x - \frac{R'}{R}y|)^{-100}(1 + |y|)^{-100}(1 + |x|)^{100}\, dy.
\end{align}
If $|x| \leq 1$, then \eqref{coeq3} is
\begin{align*}
\lsm \int_{\R^2}(1 + |y|)^{-100}\, dy \lsm 1.
\end{align*}
If $|x| > 1$, then we split \eqref{coeq3} into
\begin{align}\label{conveq2}
(\int_{|x - \frac{R'}{R}y| \leq \frac{|x|}{2}} + \int_{|x - \frac{R'}{R}y| > \frac{|x|}{2}})(1 + |x - \frac{R'}{R}y|)^{-100}(1 + |y|)^{-100}(1 + |x|)^{100}\, dy.
\end{align}
In the case of the first integral in \eqref{conveq2}, $(R'/R)|y| \geq |x| - |x - (R'/R)y| \geq |x|/2$ and hence
\begin{align*}
&\int_{|x - \frac{R'}{R}y| \leq \frac{|x|}{2}}(1 + |x - \frac{R'}{R}y|)^{-100}(1 + |y|)^{-100}(1 + |x|)^{100}\, dy\\
&\quad\quad\leq (\frac{(1 + |x|)^{100}}{(1 + (R/R')|x|/2)^{100}}\int_{\R^2}(1 + |x - \frac{R'}{R}y|)^{-100}\, dy \lsm (R'/R)^{100}(R/R')^{2} \lsm 1.
\end{align*}
In the case of the second integral in \eqref{conveq2},
\begin{align*}
\int_{|x - \frac{R'}{R}y| > \frac{|x|}{2}}(1 + |x - \frac{R'}{R}y|)^{-100}&(1 + |y|)^{-100}(1 + |x|)^{100}\, dy\\
& \leq (\frac{1 + |x|}{1 + |x|/2})^{100}\int_{\R^2}(1 + |y|)^{-100}\, dy \lsm 1.
\end{align*}
This then proves \eqref{wbconveq1}.

To prove \eqref{wbconveq2} it suffices to give a lower bound for
\begin{align*}
\frac{1}{R^2}\int_{B(0, R)}(1 + \frac{|x - y|}{R})^{-100}(1 + \frac{|x|}{R})^{100}\, dy.
\end{align*}
As before, rescaling $x$ and a change of variables in $y$ gives that it suffices to give a
lower bound for
\begin{align*}
\int_{B(0, 1)}(\frac{1 + |x|}{1 + |x - y|})^{100}\, dy \geq (\frac{1 + |x|}{2 + |x|})^{100} \gtrsim 1.
\end{align*}
The proof of the lemma for $\ow_{B(0, R)}$ is similar and so we omit the proof (details can be found on Page 14 of \cite{thesis}).
This completes the proof of the lemma.
\end{proof}

\begin{rem}\label{onedimconv}
Let $I = [-R/2, R/2]$ and $I' = [-R'/2, R'/2]$ with $0 < R' \leq R$. For $x \in \R$, let
$w_{I}(x) := (1 + \frac{|x|}{R})^{-100}$ and similarly define $w_{I'}$. The same proof as
\eqref{wbconveq1} gives that $w_{I} \ast w_{I'} \lsm R' w_{I}$ and $Rw_{I} \lsm 1_{I} \ast w_I$.
\end{rem}

\subsection{Immediate applications}
The following lemmas show how to use Lemma \ref{wbconvolve} to quickly manipulate weight functions.

\begin{lemma}\label{rot1}
Let $R$ be a rotation matrix and $|a| \lsm \delta^{-1}$. Then
\begin{align*}
w_{B(R\binom{a}{0}, \delta^{-1})}(s) \lsm w_{B(0, \delta^{-1})}(s)
\end{align*}
and the same estimate is also true with $w_{B}$ replaced with $\ow_B$.
\end{lemma}
\begin{proof}
Observe that for some sufficiently large $C$,
\begin{align*}
1_{B(R\binom{a}{0}, \delta^{-1})} \leq 1_{B(0, C\delta^{-1})} \lsm w_{B(0, C\delta^{-1})} \lsm w_{B(0, \delta^{-1})}.
\end{align*}
Convolving both sides with $w_{B(0, \delta^{-1})}$ and using Lemma \ref{wbconvolve} gives
\begin{align*}
w_{B(0, \delta^{-1})} \ast w_{B(0, \delta^{-1})} \lsm \delta^{-2}w_{B(0, \delta^{-1})}
\end{align*}
and
\begin{align*}
1_{B(R\binom{a}{0}, \delta^{-1})} \ast w_{B(0, \delta^{-1})} \gtrsim \delta^{-2}w_{B(R\binom{a}{0}, \delta^{-1})}
\end{align*}
which completes the proof of the lemma.
\end{proof}

\begin{lemma}\label{part}
Let $B$ be a square of side length $R$ and let $\mc{B}$ be a finitely overlapping cover of $B$
into squares $\Delta$ of side length $R' < R$.
Then
$$\sum_{\Delta \in \mc{B}}w_{\Delta} \lsm w_{B}.$$
This inequality remains true with $w_{\Delta}$ and $w_{B}$ replaced with $\ow_{\Delta}$ and $\ow_{B}$.
\end{lemma}
\begin{proof}
It suffices to prove the case when $B$ is centered at the origin.
Since $\mc{B}$ is a finitely overlapping cover of $B$,
$\sum_{\Delta \in \mc{B}}1_{\Delta} \lsm 1_B$.
Therefore
$$\sum_{\Delta \in \mc{B}}1_{\Delta} \ast w_{B(0, R')} \lsm 1_{B} \ast w_{B(0, R')}.$$
Lemma \ref{wbconvolve} gives that
$$R'^{2}\sum_{\Delta \in \mc{B}}w_{\Delta} \lsm \sum_{\Delta \in \mc{B}}1_{\Delta} \ast w_{B(0, R')}$$
and
$$1_{B} \ast w_{B(0, R')} \lsm R'^{2}w_{B}.$$ Rearranging then completes the proof.
\end{proof}

The following lemma allows us to upgrade from unweighted to weighted estimates.
\begin{lemma}\label{convint}
For $1 \leq p < \infty$,
\begin{align*}
\nms{f}_{L^{p}(w_{B(0, R)})}^{p} \lsm \int_{\R^2}\nms{f}_{L^{p}_{\#}(B(y, R))}^{p}w_{B(0, R)}(y)\, dy.
\end{align*}
This corollary is also true with $w_{B(0, R)}$ replaced with $\ow_{B(0, R)}$.
\end{lemma}
\begin{proof}
This estimate is from the proof of \cite[Theorem 5.1]{sg}.
The lemma follows from observing that
\begin{align*}
\int_{\R^2}\nms{f}_{L^{p}_{\#}(B(y, R))}^{p}w_{B(0, R)}(y)\, dy &= \int_{\R^2}|f(x)|^{p}(\frac{1}{R^2}1_{B(0, R)} \ast w_{B(0, R)})(x)\, dx
\end{align*}
and applying Lemma \ref{wbconvolve}.
\end{proof}

\begin{prop}\label{up}
Let $I \subset [0, 1]$ and $\mc{P}$ be a disjoint partition of $I$.
\begin{enumerate}[$(a)$]
\item Suppose for some $2 \leq p < \infty$, we have
\begin{align*}
\nms{\E_{I}g}_{L^{p}(B)} \leq C(\sum_{J \in \mc{P}}\nms{\E_{J}g}_{L^{p}(w_{B})}^{2})^{1/2}
\end{align*}
for all $g: [0, 1] \rightarrow \C$ and all squares $B$ of side length $R$. Then
\begin{align}\label{up1}
\nms{\E_{I}g}_{L^{p}(w_{B})} \lsm C(\sum_{J \in \mc{P}}\nms{\E_{J}g}_{L^{p}(w_{B})}^{2})^{1/2}
\end{align}
for all $g: [0, 1] \rightarrow \C$ and all squares $B$ of side length $R$.

\item Suppose for some $1\leq p < q < \infty$, we have
\begin{align*}
\nms{\E_{I}g}_{L^{q}_{\#}(B)} \leq C\nms{\E_{I}g}_{L^{p}_{\#}(\eta_{B}^{p})}
\end{align*}
for all $g: [0, 1] \rightarrow \C$ and all squares $B$ of side length $R$. Then
\begin{align}\label{up3}
\nms{\E_{I}g}_{L^{q}_{\#}(w_{B})} \lsm C\nms{\E_{I}g}_{L^{p}_{\#}(w_{B}^{p/q})}
\end{align}
for all $g: [0, 1] \rightarrow \C$ and all squares $B$ of side length $R$.
\end{enumerate}
The same results are also true with $w_{B}$ replaced with $\ow_{B}$.
\end{prop}
\begin{proof}
We first prove $(a)$. It suffices to prove \eqref{up1} in the case when $B$ is centered at the origin.
Corollary \ref{convint} implies that
\begin{align*}
\nms{\E_{I}g}_{L^{p}(w_{B})}^{p} &\lsm \int_{\R^2}\nms{\E_{I}g}_{L^{p}_{\#}(B(y, R))}^{p}w_{B}(y)\, dy\\
&\lsm R^{-2}C^{p}\int_{\R^2}(\sum_{J \in \mc{P}}\nms{\E_{J}g}_{L^{p}(w_{B(y, R)})}^{2})^{p/2}w_{B}(y)\, dy.
\end{align*}
Since $p \geq 2$, we can interchange the $L^{p}_{y}(w_{B})$ and $l^{2}_{J}$ norms and the
above is
\begin{align}\label{up1pf}
\lsm R^{-2}C^{p}(\sum_{J \in \mc{P}}(\int_{\R^2}\nms{\E_{J}g}_{L^{p}(w_{B(y, R)})}^{p}w_{B}(y)\, dy)^{2/p})^{p/2}.
\end{align}
Since $B$ is assumed to be centered at the origin,
\begin{align*}
\int_{\R^2}\nms{\E_{J}g}_{L^{p}(w_{B(y, R)})}^{p}w_{B}(y)\, dy = \nms{\E_{J}g}_{L^{p}(w_{B} \ast w_{B})}^{p} \lsm R^{2}\nms{\E_{J}g}_{L^{p}(w_{B})}^{p}
\end{align*}
where the inequality is an application of Lemma \ref{wbconvolve}. Inserting this into \eqref{up1pf}
gives that
\begin{align*}
\nms{\E_{I}g}_{L^{p}(w_{B})}^{p} \lsm C^{p}(\sum_{J \in \mc{P}}\nms{\E_{J}g}_{L^{p}(w_{B})}^{2})^{p/2}.
\end{align*}
Taking $1/p$ powers of both sides and using that $p \geq 1$ completes the proof of \eqref{up1}.

We now prove $(b)$. Again it suffices to prove \eqref{up3} in the case when $B$ is centered at the origin. Corollary \ref{convint}
implies that
\begin{align*}
\nms{\E_{I}g}_{L^{q}(w_{B, E})}^{q} &\lsm \int_{\R^2}\nms{\E_{I}g}_{L^{q}_{\#}(B(y, R))}^{q}w_{B}(y)\, dy\\
&\lsm C^{q}R^{-2q/p}\int_{\R^2}\nms{\E_{I}g}_{L^{p}(\eta_{B(y, R)}^{p})}^{q}w_{B}(y)\, dy\\
&= C^{q}R^{-2q/p}\int_{\R^2}(\int_{\R^2}|(\E_{I}g)(s)|^{p}\eta_{B(0, R)}(s - y)^{p}\, ds)^{q/p}w_{B}(y)\, dy.
\end{align*}
Since $q > p$, we can interchange the $L^{p}$ and $L^{q}$ norms and the above is
\begin{equation}\label{weightbernsteineq1}
\begin{aligned}
\lsm C^{q}R^{-2q/p}(\int_{\R^2}|\E_{I}g(s)|^{p}(\eta_{B}^{q} \ast w_{B})(s)^{p/q}\, ds)^{q/p}
\end{aligned}
\end{equation}
Lemma \ref{wbconvolve} gives that
\begin{align*}
\eta_{B}^{q} \ast w_{B} \lsm \exp(O(q))(w_{B}^{q} \ast w_{B}) \lsm  \exp(O(q))R^{2}w_{B}.
\end{align*}
Inserting this into \eqref{weightbernsteineq1} shows that
\begin{align*}
\nms{\E_{I}g}_{L^{q}(w_{B})}^{q} \lsm \exp(O(q))C^{q}R^{2 - 2q/p}\nms{\E_{I}g}_{L^{p}(w_{B}^{p/q})}^{q}
\end{align*}
Changing $L^{q}$ and $L^{p}$ into $L^{q}_{\#}$ and $L^{p}_{\#}$, respectively, removes the factor of $R^{2 - 2q/p}$.
Taking $1/q$ powers of both sides and using that $q \geq 1$ then completes the proof of \eqref{up3}.

Since Lemma \ref{wbconvolve} holds for $\ow_B$, the above proof also shows this proposition holds with every instance
of $w_B$ replaced with $\ow_B$. This completes the proof of the proposition.
\end{proof}

\begin{rem}\label{layercake}
Lemma \ref{convint} is not the only way to convert unweighted estimates to weighted estimates. Another approach is to prove
an unweighted estimate where $B$ is replaced by $2^{n}B$ for all $n \geq 0$ and then use that $w_{B} \sim \sum_{n \geq 0}2^{-100n}1_{2^{n}B}$
to conclude the weighted estimate.
\end{rem}

\section{Basic tools for decoupling}\label{basic}
\subsection{Parabolic rescaling}\label{parab}
We first have parabolic rescaling. We include a proof because we include a small clarification
to the proof of \cite[Proposition 7.1]{sg}, (in particular to (13) of ibid.).

We first begin with the following equivalence of decoupling constants.
\begin{rem}\label{equiv}
Let $\wt{D}_{p}(\delta)$ be the best constant such that
\begin{align*}
\nms{\E_{[0, 1]}g}_{L^{p}(B)} \leq \wt{D}_{p}(\delta)(\sum_{J \in P_{\delta}([0, 1])}\nms{\E_{J}g}_{L^{p}(\ow_{B}^{10})}^{2})^{1/2}
\end{align*}
for all $g: [0, 1] \rightarrow \C$ and squares $B$ of side length $\delta^{-2}$. Next let $K_{p}(\delta)$ be the best constant
such that
$$\nms{f}_{p} \leq K_{p}(\delta)(\sum_{J \in P_{\delta}([0, 1])}\nms{f_{\ta_J}}_{p}^{2})^{1/2}$$
for all $f$ with Fourier support in $\bigcup_{J \in P_{\delta}([0, 1])}\ta_J$ where $\ta_{J} := \{(s, s^2 + t): s \in J, |t| \leq \delta^2\}$.
Remark 5.2 of \cite{bd} (alternatively see the arguments in Sections 2.3 and 4.1 in \cite{thesis})
shows that for $2 \leq p \leq 6$, $\wt{D}_{p}(\delta) \sim K_{p}(\delta) \sim D_{p}(\delta)$ where the implied constant is absolute.

This equivalence also shows that decoupling constants are monotonic.
That is, if for two scales $\delta_s < \delta_l$, then a change of variables shows that $K_{p}(\delta_l) \leq K_{p}(\delta_s)$
(see \cite[Proposition 4.2.1]{thesis} for details). The equivalence from the previous paragraph then shows that
$D_{p}(\delta_l) \lsm D_{p}(\delta_s)$ for $2 \leq p \leq 6$ and any two scales $\delta_s$ and $\delta_l$ with $\delta_s < \delta_l$.
Note here the implied constant does not depend on $p, \delta_l$, and $\delta_s$ and is absolute.
\end{rem}

\begin{lemma}\label{parabolic_rescaling}
Let $0 < \delta < \sigma < 1$ be such that $\delta/\sigma \in \N^{-1}$. Let $I$ be an arbitrary interval
in $[0, 1]$ of length $\sigma$. Then
\begin{align*}
\nms{\E_{I}g}_{L^{p}(B)} \lsm D_{p}(\frac{\delta}{\sigma})(\sum_{J \in P_{\delta}(I)}\nms{\E_{J}g}_{L^{p}(w_B)}^{2})^{1/2}
\end{align*}
for all $g: [0, 1] \rightarrow \C$ and all squares $B \subset \R^2$ of side length $\delta^{-2}$.
\end{lemma}
\begin{proof}
It suffices to only prove the case when $B$ is centered at the origin.
Write $I = [a, a + \sigma]$. Let $T = (\begin{smallmatrix} \sigma & 2a\sigma\\0 & \sigma^{2}\end{smallmatrix})$.
We have $\nms{\E_{I}g}_{L^{p}(B)} = \sigma^{1 - 3/p}\nms{\E_{[0, 1]}g_a}_{L^{p}(T(B))}$
where $g_{a}(\eta) := g(\sigma\eta + a)$.

The corners of $B$ are given by $(\pm\delta^{-2}/2, \pm\delta^{-2}/2)$ and hence the corners of $T(B)$
are given by
\begin{equation*}
\begin{aligned}[c]
(\frac{1}{2}\sigma(1 + 2a)\delta^{-2}&, \frac{1}{2}\sigma^{2}\delta^{-2})\\
(\frac{1}{2}\sigma(1 - 2a)\delta^{-2}&, -\frac{1}{2}\sigma^{2}\delta^{-2})
\end{aligned}
\hspace{0.5in}
\begin{aligned}[c]
(-\frac{1}{2}\sigma(1 + 2a)\delta^{-2}&, -\frac{1}{2}\sigma^{2}\delta^{-2})\\
(-\frac{1}{2}\sigma(1 - 2a)\delta^{-2}&, \frac{1}{2}\sigma^{2}\delta^{-2}).
\end{aligned}
\end{equation*}
Then $T(B)$ is contained in a $O(\sigma\delta^{-2}) \times \sigma^{2}\delta^{-2}$ rectangle centered at the origin.
Cover $T(B)$ by finitely overlapping squares $\Box$ of side length $(\sigma\delta^{-1})^{2}$.
Then
\begin{align*}
\nms{\E_{[0, 1]}g_a}_{L^{p}(T(B))}^{p} &\leq \sum_{\Box}\nms{\E_{[0, 1]}g_a}_{L^{p}(\Box)}^{p}\\
&\lsm \wt{D}_{p}(\frac{\delta}{\sigma})^{p}\sum_{\Box}(\sum_{J \in P_{\delta/\sigma}([0, 1])}\nms{\E_{J}g_a}_{L^{p}(\ow_{\Box}^{10})}^{2})^{p/2}.
\end{align*}
Interchanging the $l^2$ and $l^p$ norms, undoing the change of variables, and applying Remark \ref{equiv} shows that
\begin{align*}
\nms{\E_{I}g}_{L^{p}(B)}^{p} \lsm D_{p}(\frac{\delta}{\sigma})^{p}(\sum_{J \in P_{\delta}(I)}\nms{\E_{J}g}_{L^{p}(\sum_{\Box}(\ow_{\Box})^{10}\circ T)}^{2})^{p/2}.
\end{align*}
It then remains to show that $\sum_{\Box}\ow_{\Box}(Tx)^{10} \lsm w_{B}(x)$.
We will in fact show that $\sum_{\Box}\ow_{\Box}(Tx)^{10} \lsm \ow_{B}(x)^{5}$.

Denote the center of $\Box$ by $c_{\Box} = (c_{\Box, 1}, c_{\Box, 2})$. Next let $I(a, R) = [a - R/2, a + R/2]$,
the interval centered at $a$ of length $R$. Then for some absolute constant $C$,
\begin{align}\label{ressumeq1}
\sum_{\Box}1_{I(c_{\Box, 1}, \sigma^{2}\delta^{-2})}(x_1)1_{I(c_{\Box, 2}, \sigma^{2}\delta^{-2})}(x_2) \lsm 1_{I(0, C\sigma\delta^{-2})}(x_1)1_{I(0, \sigma^{2}\delta^{-2})}(x_2).
\end{align}
Next, Remark \ref{onedimconv} shows that
\begin{align*}
(1_{I(c_{\Box, 1}, \sigma^{2}\delta^{-2})} \ast w_{I(0, \sigma^{2}\delta^{-2})}^{10})(x_1) \gtrsim \sigma^{2}\delta^{-2}w_{I(c_{\Box, 1}, \sigma^{2}\delta^{-2})}(x_1)^{10}
\end{align*}
and similarly for $c_{\Box, 2}$ and the $x_2$ variable. Combining this with \eqref{ressumeq1} shows that
\begin{align*}
\sum_{\Box}\ow_{\Box}&(Tx)^{10}\\
& \lsm (\sigma^{2}\delta^{-2})^{-2}(1_{I(0, C\sigma\delta^{-2})}^{(1)}1_{I(0, \sigma^{2}\delta^{-2})}^{(2)} \ast (w_{I(0, \sigma^{2}\delta^{-2})}^{(1)})^{10}(w_{I(0, \sigma^{2}\delta^{-2})}^{(2)})^{10})(Tx)
\end{align*}
where $1_{I}^{(1)}$ is short hand for $1_{I}(x_1)$ and similarly for $1_{I}^{(2)}$, $w_{I}^{(1)}$, and $w_{I}^{(2)}$. Applying Remark \ref{onedimconv} shows this is
\begin{align*}
\lsm ((w_{I(0, \sigma\delta^{-2})}^{(1)})^{10}(w_{I(0, \sigma^{2}\delta^{-2})}^{(2)})^{10})(Tx) = (1 + \frac{|x_1 + 2ax_2|}{\delta^{-2}})^{-1000}(1 + \frac{|x_2|}{\delta^{-2}})^{-1000}.
\end{align*}
Rescaling $x_1$ and $x_2$, it remains to show that
\begin{align*}
\frac{1 + |x_1|}{(1 + |x_1 + 2ax_2|)^{2}(1 + |x_2|)}
\end{align*}
is bounded by an absolute constant independent of $a$. This follows from considering the two cases when $|x_1 + 2ax_2| \leq |x_1|/2$
(which implies $|x_1| \lsm |x_2|$)
and $|x_1 + 2ax_2| > |x_1|/2$.
This completes the proof of the lemma.
\end{proof}

Parabolic rescaling immediately implies the following corollary which allows us to reduce to showing good estimates for $D_{p}(\delta)$ along
a lacunary sequence of scales.
\begin{cor}\label{almostmult}
Suppose $\delta_1, \delta_2 \in \N^{-1}$. Then $D_{p}(\delta_1 \delta_2)\lsm D_{p}(\delta_1)D_{p}(\delta_2)$.
\end{cor}
In applications, this corollary is a useful cheap substitute to bypass the need to prove monotonicity of the decoupling constant
and yields qualitatively the same results.

\subsection{Bilinearization}
We now define the bilinear constant.
Let $b$ be an integer $\geq 1$.
Suppose $\delta \in \N^{-1}$ and $\nu \in \N^{-1} \cap (0, 1/100)$ were such that $\nu^{b}\delta^{-1} \in \N$.
Let $\mb{M}_{p, b}(\delta, \nu)$ be the best constant such that
\begin{align}\label{mb_bds}
\begin{aligned}
\avg{\Delta \in P_{\nu^{-b}}(B)}&(\sum_{J \in P_{\nu^{b}}(I)}\nms{\E_{J}g}_{L^{2}_{\#}(\ow_\Delta)}^{2})^{\frac{p}{4}}(\sum_{J' \in P_{\nu^{b}}(I')}\nms{\E_{J'}g}_{L^{2}_{\#}(\ow_\Delta)}^{2})^{\frac{p}{4}}\\
& \leq \mb{M}_{p, b}(\delta, \nu)^{p}(\sum_{J \in P_{\delta}(I)}\nms{\E_{J}g}_{L^{p}_{\#}(w_B)}^{2})^{\frac{p}{4}}(\sum_{J' \in P_{\delta}(I')}\nms{\E_{J'}g}_{L^{p}_{\#}(w_B)}^{2})^{\frac{p}{4}}
\end{aligned}
\end{align}
for all squares $B$ of side length $\delta^{-2}$, $g: [0, 1] \rightarrow \C$, and all intervals $I, I' \in P_{\nu}([0, 1])$ which are $\nu$-separated.
Note that the left hand side of \eqref{mb_bds} is the same as $A_{p}(q, B^r, q)$ constructed in Section 10 of \cite{sg}.

From parabolic rescaling, the bilinear decoupling constant is controlled by the decoupling constant.
\begin{lemma}\label{bil_lin}
If $\delta$ and $\nu$ were such that $\nu^{b}\delta^{-1} \in \N^{-1}$, then
\begin{align*}
\mb{M}_{p, b}(\delta, \nu) \lsm D_{p}(\frac{\delta}{\nu^b}).
\end{align*}
\end{lemma}
\begin{proof}
Since $p \geq 2$, moving from $L^{2}_{\#}$ up to $L^{p}_{\#}$ and then
applying H\"{o}lder in the average over $\Delta$ controls the left hand side of \eqref{mb_bds} by
\begin{align*}
(\avg{\Delta \in P_{\nu^{-b}}(B)}(\sum_{J \in P_{\nu^b}(I)}\nms{\E_{J}g}_{L^{p}_{\#}(\ow_\Delta)}^{2})^{p/2})^{1/2}(\avg{\Delta \in P_{\nu^{-b}}(B)}(\sum_{J' \in P_{\nu^b}(I')}\nms{\E_{J'}g}_{L^{p}_{\#}(\ow_\Delta)}^{2})^{p/2})^{1/2}.
\end{align*}
Since $p \geq 2$, interchanging the $l^2$ and $l^p$ norms and then using that $\sum_{\Delta}\ow_{\Delta} \lsm \ow_B \lsm w_B$ controls the above by
\begin{align*}
(\sum_{J \in P_{\nu^b}(I)}\nms{\E_{J}g}_{L^{p}_{\#}(w_B)}^{2})^{p/4}(\sum_{J' \in P_{\nu^b}(I')}\nms{\E_{J'}g}_{L^{p}_{\#}(w_B)}^{2})^{p/4}.
\end{align*}
Applying parabolic rescaling then completes the proof.
\end{proof}

Lemma \ref{bil_lin} shows that if we have good bounds for the decoupling constant, then we have good bounds for the bilinear decoupling constant.
The next lemma shows that the opposite is also true.
\begin{lemma}\label{reduction}
Suppose $\delta$ and $\nu$ were such that $\nu\delta^{-1} \in \N$. Then
\begin{align*}
D_{p}(\delta) \lsm D_{p}(\frac{\delta}{\nu}) + \nu^{-1}\mb{M}_{p, 1}(\delta, \nu).
\end{align*}
\end{lemma}
\begin{proof}
This proof is the same as that of Lemmas 2.5, 4.5, and 5.2 of \cite{ef2d}.
Observe that
\begin{align*}
\|\E_{[0, 1]}&g\|_{L^{p}(B)}\\& \lsm (\sum_{\st{I, I' \in P_{\nu}([0, 1])\\d(I, I') \lsm \nu}}\nms{|\E_{I}g||\E_{I'}g|}_{L^{p/2}(B)})^{1/2} +\nu^{-1}\max_{\st{I, I'\in P_{\nu}([0, 1])\\d(I, I') \gtrsim \nu}}\nms{|\E_{I}g||\E_{I'}g|}_{L^{p/2}(B)}^{1/2}.
\end{align*}
For the first term on the right hand side, we apply H\"{o}lder and then parabolic rescaling. For the second term we observe that
\begin{align*}
\frac{1}{|B|}\int_{B}|\E_{I}g|^{p/2}|\E_{I'}g|^{p/2} &= \avg{\Delta \in P_{\nu^{-1}}(B)}\frac{1}{|\Delta|}\int_{\Delta}|\E_{I}g|^{p/2}|\E_{I'}g|^{p/2}\\
&\leq\avg{\Delta \in P_{\nu^{-1}}(B)}\nms{\E_{I}g}_{L^{p/2}_{\#}(\Delta)}^{p/2}\nms{\E_{I'}g}_{L^{\infty}(\Delta)}^{p/2}\\
&\lsm \avg{\Delta \in P_{\nu^{-1}}(B)}\nms{\E_{I}g}_{L^{2}_{\#}(\ow_{\Delta})}^{p/2}\nms{\E_{I'}g}_{L^{2}_{\#}(\ow_{\Delta})}^{p/2}
\end{align*}
where in the last inequality we have used reverse H\"{o}lder (see Lemma \ref{bridge} below). Now we apply the definition of $\mb{M}_{p, 1}(\delta, \nu)$.
\end{proof}

\subsection{Reverse H\"{o}lder and $l^2 L^2$ decoupling}
Finally we have the reverse H\"{o}lder and $l^2 L^2$ decoupling inequalities which we omit
the proofs of save for a remark in Lemma \ref{bridge} on why the implied constants are independent of $p$ and $q$

\begin{lemma}\label{bridge}
Let $1 \leq p < q \leq \infty$, $J \subset [0, 1]$ with $\ell(J) = 1/R$ and $B \subset \R^2$ a square with side length $R \geq 1$.
If $q < \infty$, then
\begin{align}\label{bridgeres1}
\nms{\E_{J}g}_{L^{q}_{\#}(w_{B})} \lsm \nms{\E_{J}g}_{L^{p}_{\#}(w_{B}^{p/q})}.
\end{align}
If $q = \infty$, then
\begin{align}\label{bridgeres2}
\sup_{x \in B}|(\E_{J}g)(x)| \lsm \nms{\E_{J}g}_{L^{p}_{\#}(w_{B})}.
\end{align}
This statement is also true with $w_B$ replaced with $\ow_B$.
\end{lemma}
\begin{proof}
Following the proof of \cite[Corollary 4.3]{sg}, for $1 \leq p < q \leq \infty$ we have
\begin{align*}
\nms{\E_{J}g}_{L^{q}(B)} \lsm R^{2(\frac{1}{q} - \frac{1}{p})}\nms{\E_{J}g}_{L^{p}(\eta_{B}^{p})}
\end{align*}
where the implied constant is absolute (to see this, observe that in the proof of \cite[Corollary 4.3]{sg},
$\nms{\wh{\ta_Q}}_{L^{r}(\R^n)} \lsm R^{-n/r'}$ where the implied constant is absolute).
Since the application of Lemma \ref{up} does not introduce new dependencies, the implied constant in \eqref{bridgeres1}
and \eqref{bridgeres2} are absolute.
\end{proof}

\begin{lemma}\label{l2decouplingleq}
Let $J \subset [0, 1]$ be an interval of length $\geq 1/R$ such that $|J|R \in \N$. Then for each square $B \subset \R^2$ with side length $R$,
\begin{align*}
\nms{\E_{J}g}_{L^{2}(w_{B})}^{2} \lsm \sum_{J' \in P_{1/R}(J)}\nms{\E_{J'}g}_{L^{2}(w_{B})}^{2}.
\end{align*}
This statement is also true with $w_B$ replaced with $\ow_B$.
\end{lemma}
\begin{proof}
See \cite[Proposition 6.1]{sg} for a proof.
\end{proof}

\section{Ball Inflation}\label{ball}
We prove the following ball inflation lemma. We follow the proof as in \cite[Theorem 9.2]{sg} except we pay close
attention to the dependence on $p$ and the power of $\log$ we obtain.
\begin{lemma}\label{ballinflation}
Let $p > 4$, $b \geq 1$, and $I_1, I_2$ are $\nu$-separated intervals of length $\nu$.
Then for any square $\Delta'$ of side length $\nu^{-2b}$, we have
\begin{equation}\label{ballinflationeq1}
\begin{aligned}
\avg{\Delta \in P_{\nu^{-b}}(\Delta')}&\geom(\sum_{J \in P_{\nu^b}(I_i)}\nms{\E_{J}g}_{L^{p/2}_{\#}(\ow_{\Delta})}^{2})^{p/2}\\
&\leq \exp(O(p))\nu^{-1}(\log\frac{1}{\nu^b})^{p/2}\geom(\sum_{J \in P_{\nu^b}(I_i)}\nms{\E_{J}g}_{L^{p/2}_{\#}(\ow_{\Delta'})}^{2})^{p/2}
\end{aligned}
\end{equation}
\end{lemma}
We will first prove a version of Lemma \ref{ballinflation} where we additionally assume that all
the $\nms{\E_{J}g}$ are of comparable size (for each $I_i$).

\begin{lemma}\label{dyadic}
Let everything be as defined in Lemma \ref{ballinflation}.
Furthermore, let $\mc{F}_1$ be a collection of intervals in $P_{\nu^b}(I_1)$ such that for each pair of intervals $J, J' \in \mc{F}_1$, we have
\begin{align}\label{comp}
\frac{1}{2} < \frac{\nms{\E_{J}g}_{L^{p/2}_{\#}(\ow_{\Delta'})}}{\nms{\E_{J'}g}_{L^{p/2}_{\#}(\ow_{\Delta'})}} \leq 2.
\end{align}
Similarly define $\mc{F}_2$. Then
\begin{align}\label{logball}
\begin{aligned}
\avg{\Delta \in P_{\nu^{-b}}(\Delta')}\geom(\sum_{J \in \mc{F}_i}&\nms{\E_{J}g}_{L^{p/2}_{\#}(\ow_{\Delta})}^{2})^{p/2}\\
& \lsm \exp(O(p))\nu^{-1}\geom(\sum_{J \in \mc{F}_i}\nms{\E_{J}g}_{L^{p/2}_{\#}(\ow_{\Delta'})}^{2})^{p/2}.
\end{aligned}
\end{align}
\end{lemma}
\begin{proof}
For each $J \in P_{\nu^b}(I_i)$ centered at $c_J$, cover $\Delta'$ by a set $\mc{T}_J$ of mutually parallel nonoverlapping boxes $P_J$ with dimension $\nu^{-b} \times \nu^{-2b}$
with longer side pointing in the direction $(-2c_J, 1)$. Note that any $\nu^{-b} \times \nu^{-2b}$ box outside $4\Delta'$ cannot cover $\Delta'$ itself.
Thus we may assume that all the boxes in $\mc{T}_J$ are contained in $4\Delta'$. Finally, let $P_{J}(x)$ denote the box in $\mc{T}_J$ containing $x$
and let $2P_{J}$ be the $2\nu^{-b} \times 2\nu^{-2b}$ box having the same center and orientation as $P_J$.

Since $p > 4$, H\"{o}lder's inequality yields that
\begin{align*}
(\sum_{J \in \mc{F}_i}\nms{\E_{J}g}_{L^{p/2}_{\#}(\ow_{\Delta})}^{2})^{p/2} \leq (\sum_{J \in \mc{F}_i}\nms{\E_{J}g}_{L^{p/2}_{\#}(\ow_{\Delta})}^{p/2})^{2}|\mc{F}_i|^{p/2 - 2}.
\end{align*}
Note that the reverse inequality (with an extra multiplicative factor of $\exp(O(p))$) is also true because of our assumption on $\mc{F}_i$.
The left hand side of \eqref{logball} is then bounded above by
\begin{align}\label{pentstep}
(\prod_{i = 1}^{2}|\mc{F}_i|^{p/4 - 1})\avg{\Delta \in P_{\nu^{-b}}(\Delta')}\prod_{i = 1}^{2}(\sum_{J \in \mc{F}_i}\nms{\E_{J}g}_{L^{p/2}_{\#}(\ow_{\Delta})}^{p/2}).
\end{align}
For $x \in 4\Delta'$, define
\begin{align}\label{hjdef}
\begin{aligned}
H_{J}(x) :=
\begin{cases}
\sup_{y \in 2P_{J}(x)}\nms{\E_{J}g}_{L^{p/2}_{\#}(\ow_{B(y, \nu^{-b})})}^{p/2} & \text{if } x \in \bigcup_{P_J \in \mc{T}_J}P_J\\
0 & \text{if } x \in 4\Delta' \bs \bigcup_{P_J \in \mc{T}_J}P_J.
\end{cases}
\end{aligned}
\end{align}
For each $x \in \Delta$, observe that $\Delta \subset 2P_{J}(x)$.
Therefore for each $x \in \Delta$, $c_{\Delta} \in 2P_{J}(x)$ and hence
\begin{align}\label{ejgj}
\nms{\E_{J}g}_{L^{p/2}_{\#}(\ow_{\Delta})}^{p/2} \leq H_{J}(x)
\end{align}
for all $x \in \Delta$.
Thus
\begin{align}\label{firstbound}
\avg{\Delta \in P_{\nu^{-b}}(\Delta')}\prod_{i = 1}^{2}(\sum_{J \in \mc{F}_i}\nms{\E_{J}g}_{L^{p/2}_{\#}(\ow_{\Delta})}^{p/2})&= \sum_{\st{J_1 \in \mc{F}_1\\J_2 \in \mc{F}_2}}\avg{\Delta \in P_{\nu^{-b}}(\Delta')}\nms{\E_{J_1}g}_{L^{p/2}_{\#}(\ow_{\Delta})}^{p/2}\nms{\E_{J_2}g}_{L^{p/2}_{\#}(\ow_{\Delta})}^{p/2}\frac{1}{|\Delta|}\int_{\Delta}\, dx\nonumber\\
&\leq \sum_{\st{J_1 \in \mc{F}_1\\J_2 \in \mc{F}_2}}\frac{1}{|\Delta'|}\int_{\Delta'}H_{J_1}(x)H_{J_2}(x)\, dx
\end{align}
where the last inequality we have used \eqref{ejgj}.
By how $H_J$ is defined, $H_J$ is constant on each $P_J \in \mc{T}_J$. That is, for each $x \in \bigcup_{P_J \in \mc{T}_J}P_J$,
\begin{align*}
H_{J}(x) = \sum_{P_J \in \mc{T}_J}c_{P_J}1_{P_J}(x)
\end{align*}
for some constants $c_{P_J} \geq 0$. Then
\begin{align*}
\frac{1}{|\Delta'|}\int_{\Delta'}H_{J_1}(x)H_{J_2}(x)\, dx &= \frac{1}{|\Delta'|}\sum_{\st{P_{J_1} \in \mc{T}_{J_1}\\P_{J_2} \in \mc{T}_{J_2}}}c_{P_{J_1}}c_{P_{J_2}}|(P_{J_1} \cap P_{J_2}) \cap \Delta'|\\
&\leq \frac{1}{|\Delta'|}\sum_{\st{P_{J_1} \in \mc{T}_{J_1}\\P_{J_2} \in \mc{T}_{J_2}}}c_{P_{J_1}}c_{P_{J_2}}|P_{J_1} \cap P_{J_2}|
\end{align*}
where the last inequality is because $c_{P_J} \geq 0$ for all $P_J$. Since $|P_J| = \nu^{-3b}$ we also have
\begin{align}\label{kakeyapar}
\frac{1}{|\Delta'|}\int_{4\Delta'}H_{J}(x)\, dx = \frac{1}{|\Delta'|}\int_{\bigcup_{P_J \in \mc{T}_J}P_J}\sum_{P_J \in \mc{T}_J}c_{P_J}1_{P_J}(x)\, dx = \nu^b\sum_{P_J \in \mc{T}_J}c_{P_J}.
\end{align}

Recall that $J_1 \in \mc{F}_1\subset P_{\nu^b}(I_1)$ and $J_2 \in \mc{F}_2 \subset P_{\nu^b}(I_2)$.
Since $I_1$ and $I_2$ are $\nu$-separated, so are $J_1$ and $J_2$ which implies
$|P_{J_1} \cap P_{J_2}| \lsm \nu^{-2b - 1}$.
Applying this  gives
\begin{align*}
\frac{1}{|\Delta'|}\sum_{\st{P_{J_1} \in \mc{T}_{J_1}\\P_{J_2} \in \mc{T}_{J_2}}}&c_{P_{J_1}}c_{P_{J_2}}|P_{J_1} \cap P_{J_2}|\\
&\lsm \frac{\nu^{-2b - 1}}{|\Delta'|}\prod_{i = 1}^{2}(\frac{\nu^{-b}}{|\Delta'|}\int_{4\Delta'}H_{J_i}(x)\, dx)\lsm  \nu^{-1}\prod_{i = 1}^{2}\frac{1}{|4\Delta'|}\int_{4\Delta'}H_{J_i}(x)\, dx.
\end{align*}
Therefore \eqref{firstbound} is bounded above by
\begin{equation}\label{secondbound}
\nu^{-1}\prod_{i = 1}^{2}(\sum_{J \in \mc{F}_i}\frac{1}{|4\Delta'|}\int_{4\Delta'}H_{J}(x)\, dx).
\end{equation}
We now apply Lemma \ref{hjest}, proven later, to \eqref{secondbound} which shows that it is
\begin{align*}
\lsm \exp(O(p))\nu^{-1}\prod_{i = 1}^{2}(\sum_{J \in \mc{F}_i}\nms{\E_{J}g}_{L^{p/2}_{\#}(\ow_{\Delta'})}^{p/2}).
\end{align*}
Thus \eqref{pentstep} is bounded above by
\begin{align*}
\exp(O(p))\nu^{-1}(\prod_{i = 1}^{2}|\mc{F}_i|^{p/4 - 1})\prod_{i = 1}^{2}(\sum_{J \in \mc{F}_i}\nms{\E_{J}g}_{L^{p/2}_{\#}(\ow_{\Delta'})}^{p/2}).
\end{align*}
Since the intervals in $\mc{F}_i$ satisfy \eqref{comp}, the above is
\begin{align*}
\lsm \exp(O(p))\nu^{-1}\geom(\sum_{J \in \mc{F}_i}\nms{\E_{J}g}_{L^{p/2}_{\#}(\ow_{\Delta'})}^{2})^{p/2}
\end{align*}
which completes the proof of Lemma \ref{dyadic} assuming we can prove Lemma \ref{hjest} which we do below.
\end{proof}

We let $\Psi: \R \rightarrow \R$ be a Schwartz function which is equal to 1 on $[-1, 1]$ and vanishes
outside $[-3, 3]$ such that for integer $k \geq 0$,
\begin{align}\label{moments}
|\int_{\R}t^k \Psi(t)e^{2\pi i tx}\, dt| \lsm \exp(O(k))(1 + |x|)^{-20000}.
\end{align}
An explicit example of $\Psi$ is constructed in \cite[Lemma 2.2.10]{thesis}.

\begin{lemma}\label{hjest}
Let $H_J$ be as defined in \eqref{hjdef}.
Then
\begin{align*}
\frac{1}{|4\Delta'|}\int_{4\Delta'}H_{J}(x)\, dx \lsm \exp(O(p))\nms{\E_{J}g}_{L^{p/2}_{\#}(\ow_{\Delta'})}^{p/2}.
\end{align*}
\end{lemma}
\begin{proof}
This is the inequality proven in (29) of \cite{sg} without explicit constants. We follow their proof, this time paying attention to the implied constants.

Fix arbitrary $J \subset [0, 1]$ of length $\nu^{b}$ and center $c_J$.
For $x \in \bigcup_{P_J \in \mc{T}_J}P_J = \supp H_J \subset 4\Delta'$,
fix arbitrary $y \in 2P_{J}(x)$. Note that $2P_{J}(x)$ is a rectangle of dimension $2\nu^{-b} \times 2\nu^{-2b}$
with the longer side pointing in the direction of $(-2c_J, 1)$.

Let $R_J$ be the rotation matrix such that $R_J^{-1}$ applied to $2P_{J}(x)$ gives an axis parallel rectangle
of dimension $2\nu^{-b} \times 2\nu^{-2b}$ with the longer side pointing in the vertical direction.
Since $y \in 2P_{J}(x)$, we can write $$R_{J}^{-1}y = R_{J}^{-1}x + \ov{y}$$
where $|\ov{y}_1| \leq 2\nu^{-b}$ and $|\ov{y}_2| \leq 2\nu^{-2b}$.
We then have
\begin{align*}
\nms{\E_{J}g}_{L^{p/2}(\ow_{B(y, \nu^{-b})})}^{p/2} = \int_{\R^2}|(\E_{J}g)(s)|^{p/2}\ow_{B(x + R_{J}\ov{y}, \nu^{-b})}(s)\, ds
\end{align*}
Writing $\ov{y} = (\ov{y}_1, 0)^{T} + (0, \ov{y}_2)^{T}$ and a change of variables gives that the above is equal to
\begin{align}\label{ejgr2}
&\int_{\R^2}|(\E_{J}g)(s + x + R_{J}(0, \ov{y}_2)^{T})|^{p/2}\ow_{B(R_{J}(\ov{y}_1, 0)^{T}, \nu^{-b})}(s)\, ds\nonumber\\
&\lsm \int_{\R^2}|(\E_{J}g)(s + x + R_{J}(0, \ov{y}_2)^{T})|^{p/2}\ow_{B(0, \nu^{-b})}(s)\, ds
\end{align}
where here we have used Lemma \ref{rot1}.
Observe that
\begin{align*}
|(\E_{J}g)(s + x + R_{J}(0, \ov{y}_2)^{T})| = |\int_{\R^2}\wh{\E_{J}g}(\ld)e(\ld\cdot (s + x))e(\ld \cdot R_{J}(0, \ov{y}_2)^{T})\, d\ld|.
\end{align*}
Since $R_J$ is a rotation matrix, a change of variables gives that the above is equal to
\begin{align}\label{taylorprev}
|\int_{\R^2}\wh{\E_{J}g}(R_{J}\ld)e(\ld \cdot R_{J}^{-1}(s + x))e(\ld \cdot (0, \ov{y}_2)^{T})\, d\ld|
\end{align}
Writing $$e(\ld \cdot(0, \ov{y}_2)^T) = e((\ld_2 - c_{J}^2)\ov{y}_2)e(c_{J}^2\ov{y}_2) = e(c_{J}^2\ov{y}_2)\sum_{k = 0}^{\infty}\frac{(2\pi i)^k\ov{y}_{2}^{k}}{k!}(\ld_2 - c_{J}^2)^k$$
and using that $|\ov{y}_2| \leq 2\nu^{-2b}$ shows that \eqref{taylorprev} is
\begin{align*}
\leq \sum_{k = 0}^{\infty}\frac{(4\pi)^k}{k!}|\int_{\R^2}\wh{\E_{J}g}(R_{J}\ld)e(\ld \cdot R_{J}^{-1}(s + x))(\frac{\ld_2 - c_{J}^2}{\nu^{2b}})^{k}\, d\ld|.
\end{align*}
Applying the change of variables $\eta = \ld - (c_J, c_{J}^2)$ gives that the above is
\begin{align}\label{ejgr3}
\leq \sum_{k = 0}^{\infty}\frac{30^k}{k!}|\int_{\R^2}\wh{\E_{J}g}(R_{J}(\eta + (c_J, c_{J}^2)))e(\eta \cdot R_{J}^{-1}(s + x))(\frac{\eta_2 }{2\nu^{2b}})^k\, d\eta|.
\end{align}
Note that $\wh{\E_{J}g}(R_{J}(\eta + (c_J, c_{J}^2)))$ is supported in a $4\nu^b \times 4\nu^{2b}$ box centered at the origin
pointing in the horizontal direction. Thus we may insert the Schwartz function $\Psi$ defined before the statement of Lemma \ref{hjest}.
Then \eqref{ejgr3} becomes
\begin{align*}
\sum_{k = 0}^{\infty}\frac{30^k}{k!}|\int_{\R^2}\wh{\E_{J}g}(R_{J}(\eta + (c_J, c_{J}^2)))e(\eta \cdot R_{J}^{-1}(s + x))(\frac{\eta_2}{2\nu^{2b}})^k \Psi(\frac{\eta_1}{2\nu^{b}})\Psi(\frac{\eta_2}{2\nu^{2b}})\, d\eta|.
\end{align*}

Let $\Phi_{k}(t) := t^k \Psi(t)$ and let
$$(M_{k} f)(x) = \int_{\R^2}\wh{f}(R_{J}(\eta + (c_J, c_{J}^2)))e(\eta \cdot x)\Psi(\frac{\eta_1}{2\nu^{b}})\Phi_k(\frac{\eta_2}{2\nu^{2b}})\, d\eta.$$
Thus we have shown that
\begin{align*}
|(\E_{J}g)(s + x + R_{J}(0, \ov{y}_2)^{T})| \leq \sum_{k = 0}^{\infty}\frac{30^k}{k!}|(M_{k}\E_{J}g)(R_{J}^{-1}(s + x))|
\end{align*}
and combining this with \eqref{ejgr2} gives that for $x \in \bigcup_{P_{J} \in \mc{T}_{J}}P_{J}$ and $y \in 2P_{J}(x)$,
\begin{align*}
\nms{\E_{J}g}_{L^{p/2}_{\#}(\ow_{B(y, \nu^{-b})})}^{p/2} \lsm \nu^{2b}\int_{\R^2}(\sum_{k = 0}^{\infty}\frac{30^k}{k!}|(M_{k}\E_{J}g)(R_{J}^{-1}(s + x))|)^{p/2}\ow_{B(0, \nu^{-b})}(s)\, ds.
\end{align*}
Thus
\begin{align}\label{hjest1}
\frac{1}{|4\Delta'|}&\int_{4\Delta'}H_{J}(x)\, dx\nonumber\\
 &\lsm \nu^{6b}\int_{\bigcup_{P_J}P_J}\int_{\R^2}(\sum_{k = 0}^{\infty}\frac{30^k}{k!}|(M_{k}\E_{J}g)(R_{J}^{-1}(s + x))|)^{p/2}\ow_{B(0, \nu^{-b})}(s)\, ds\, dx\nonumber\\
&\lsm \nu^{6b}\int_{\R^2}(\sum_{k = 0}^{\infty}\frac{30^k}{k!}|(M_{k}\E_{J}g)(u)|)^{p/2}(\int_{4\Delta'}\ow_{B(x, \nu^{-b})}(R_{J}u)\, dx)\, du.
\end{align}
As $1_{4\Delta'} \lsm \ow_{4\Delta'} \lsm \ow_{\Delta'}$,
\begin{align*}
\int_{4\Delta'}w_{B(x, \nu^{-b})}(R_{J}u)\, dx &= (1_{4\Delta'} \ast \ow_{B(0, \nu^{-b})})(R_{J}u)\\
& \lsm (\ow_{\Delta'} \ast \ow_{B(0, \nu^{-b})})(R_{J}u) \lsm \nu^{-2b}\ow_{\Delta'}(R_{J}u).
\end{align*}
It follows that \eqref{hjest1} is bounded by
\begin{align}\label{hjl1norm}
\nu^{4b}(\sum_{k = 0}^{\infty}\frac{30^{k}}{k!}\nms{M_{k}\E_{J}g \circ R_{J}^{-1}}_{L^{p/2}(\ow_{\Delta'})})^{p/2}.
\end{align}
Inserting an extra $e(R_{J}(c_J, c_{J}^{2})^{T}\cdot s)$ and applying a change of variables gives
\begin{align*}
|(M_{k}\E_{J}g)(R_{J}^{-1}s)| &= |\int_{\R^2}\wh{\E_{J}g}(R_{J}(\eta + (c_J, c_{J}^{2})))e(R_{J}\eta \cdot s)\Psi(\frac{\eta_1}{2\nu^{b}})\Phi_k(\frac{\eta_2}{2\nu^{2b}})\, d\eta|\\
&= |\int_{\R^2}\wh{\E_{J}g}(\gamma)e(\gamma \cdot s)\wh{m_k}(\gamma)\, d\gamma|
\end{align*}
where
\begin{align*}
\wh{m_k}(\gamma) = \Psi(\frac{\gamma_1 \cos\ta_J + \gamma_2 \sin\ta_J - c_J}{2\nu^{b}})\Phi_k(\frac{\gamma_2\cos\ta_J - \gamma_1 \sin\ta_J - c_{J}^2}{2\nu^{2b}})
\end{align*}
and
$\ta_J$ is such that $R_J = (\begin{smallmatrix}\cos \ta_J & -\sin\ta_J\\\sin\ta_J & \cos\ta_J\end{smallmatrix})$.
Then $|M_{k}\E_{J}g \circ R_{J}^{-1}| = |\E_{J}g \ast m_{k}| \leq |\E_{J}g| \ast |m_{k}|$
and H\"{o}lder's inequality implies
$$(|\E_{J}g| \ast |m_{k}|)^{p/2} \leq (|\E_{J}g|^{p/2} \ast |m_{k}|)\nms{m_{k}}_{L^{1}}^{p/2 - 1}.$$
Therefore
\begin{align}\label{mkrjinv}
\nms{M_{k}\E_{J}g \circ R_{J}^{-1}}_{L^{p/2}(\ow_{\Delta'})} \leq \nms{m_{k}}_{L^{1}(\R^2)}^{1 - 2/p}\nms{\E_{J}g}_{L^{p/2}(\ow_{\Delta'} \ast |m_k|(-\cdot))}
\end{align}
where here $|m_k|(-\cdot)$ is the function $|m_k|(-x)$.
Since $\Phi$ and $\Psi$ are both Schwartz functions, our goal will be to use the rapid decay to show that $|m_k| \lsm \ow_{\Delta'}$.
A change of variables gives
\begin{align*}
|m_k(x)| &= |\int_{\R^2}\wh{m_{k}}(\gamma)e^{2\pi i x\cdot \gamma}\, d\gamma|\\
&=|\int_{\R^2}\Psi(\frac{(R_{J}^{-1}\gamma)_1 - c_J}{2\nu^b})\Phi_{k}(\frac{(R_{J}^{-1}\gamma)_2 - c_{J}^{2}}{2\nu^{2b}})e^{2\pi i x\cdot \gamma}\, d\gamma|\\
&= 4\nu^{3b}|\int_{\R}\Psi(w_1)e^{2\pi i (R_{J}^{-1}x)_{1}(2\nu^{b} w_1)}\, dw_1\int_{\R}\Phi_{k}(w_2)e^{2\pi i (R_{J}^{-1}x)_2(2\nu^{2b} w_2)}\, dw_2|.
\end{align*}

By \eqref{moments},
\begin{align*}
|\int_{\R}\Psi(w_1)e^{2\pi i (R_{J}^{-1}x)_{1}(2\nu^{b} w_1)}\, dw_1| \lsm (1 + 2\nu^{b}|(R_{J}^{-1}x)_1|)^{-20000}
\end{align*}
and
\begin{align*}
|\int_{\R}\Phi_{k}(w_2)e^{2\pi i (R_{J}^{-1}x)_2(2\nu^{2b} w_2)}\, dw_2| \lsm \exp(O(k))(1 + 2\nu^{2b}|(R_{J}^{-1}x)_2|)^{-20000}.
\end{align*}
Therefore
\begin{align}\label{rotapp}
|m_{k}(x)| \lsm \nu^{3b}\exp(O(k))(1 + \frac{|(R_{J}^{-1}x)_1|}{\nu^{-b}})^{-20000}(1 + \frac{|(R_{J}^{-1}x)_2|}{\nu^{-2b}})^{-20000}.
\end{align}
Thus we have
\begin{align}\label{mkl1}
\nms{m_k}_{L^{1}(\R^2)}^{1 - 2/p} \lsm \exp(O(k))^{1 - 2/p}.
\end{align}
Next we claim that \eqref{rotapp} is
\begin{align}\label{rotapp2}
\lsm \nu^{3b}\exp(O(p))(1 + \frac{|x_1|}{\nu^{-b}})^{-2000}(1 + \frac{|x_2|}{\nu^{-2b}})^{-2000}.
\end{align}
To see this, apply the layer cake decomposition mentioned in Remark \ref{layercake} to the $x_1$ and $x_2$ variables separately. Next
consider a $2^{n}\nu^{-b} \times 2^{n}\nu^{-2b}$ rectangle centered at the origin where the long side makes
an angle $\ta_J$ with $\ta_J \in (0, \tan^{-1}(2))$ (the angle $\tan^{-1}(2)$ comes from if $c_J = 1$).
Such a rectangle is contained in an axis-parallel $C2^{n}\nu^{-b} \times C2^{n}\nu^{-2b}$ rectangle for some sufficiently
large absolute constant $C$. Applying the layer cake decomposition again then shows that \eqref{rotapp} is controlled
by \eqref{rotapp2}.

Remark \ref{onedimconv} then implies that
\begin{align*}
\ow_{\Delta'} \ast |m_k|(-\cdot) \lsm \exp(O(k))\ow_{\Delta'}
\end{align*}
and hence
\begin{align*}
\nms{\E_{J}g}_{L^{p/2}(\ow_{\Delta'} \ast |m_k|(-\cdot))} \lsm (\exp(O(k)))^{2/p}\nms{\E_{J}g}_{L^{p/2}(\ow_{\Delta'})}.
\end{align*}
Combining this with \eqref{hjl1norm}, \eqref{mkrjinv}, and \eqref{mkl1} shows that
\begin{align*}
\frac{1}{|4\Delta'|}\int_{4\Delta'}H_{J}(x)\, dx&\lsm \exp(O(p))\nu^{4b}(\sum_{k = 0}^{\infty}\frac{\exp(O(k))}{k!}\nms{\E_{J}g}_{L^{p/2}(\ow_{\Delta'})})^{p/2}\\
& \lsm \exp(O(p))\nms{\E_{J}g}_{L^{p/2}_{\#}(\ow_{\Delta'})}^{p/2}
\end{align*}
which completes the proof of the lemma.
\end{proof}

\begin{proof}[Proof of Lemma \ref{ballinflation}]
For $i = 1, 2$, let $$M_i := \max_{J \in P_{\nu^{b}}(I_i)}\nms{\E_{J}g}_{L^{p/2}_{\#}(\ow_{\Delta'})}.$$
For each $i = 1, 2$, let $\mc{F}_{i, 0}$ denote the set of intervals $J' \in P_{\nu^{b}}(I_i)$ such that
$$\nms{\E_{J'}g}_{L^{p/2}_{\#}(w_{\Delta'})} \leq \nu^{3b}M_i$$
and partition the remaining intervals in $P_{\nu^{b}}(I_i)$ into $\lceil\log_{2}(\nu^{-3b})\rceil$
many classes $\mc{F}_{i, k}$ (with $k = 1, 2, \ldots, \lceil\log_{2}(\nu^{-3b})\rceil$) such that
\begin{align*}
2^{k - 1}\nu^{3b}M_i < \nms{\E_{J'}g}_{L^{p/2}_{\#}(\ow_{\Delta'})} \leq 2^{k}\nu^{3b}M_i
\end{align*}
for all $J' \in \mc{F}_{i, k}$.
Note that $\mc{F}_{i, k}$ satisfies the hypothesis \eqref{comp} given in Lemma \ref{dyadic}.
For $0\leq k, l \leq \lceil \log_{2}(\nu^{-3b})\rceil$, let
\begin{align*}
F_{\Delta}(k, l) := (\sum_{J \in \mc{F}_{1, k}}\nms{\E_{J}g}_{L^{p/2}_{\#}(\ow_{\Delta})}^{2})^{p/4}(\sum_{J \in \mc{F}_{2, l}}\nms{\E_{J}g}_{L^{p/2}_{\#}(\ow_{\Delta})}^{2})^{p/4}.
\end{align*}

The left hand side of \eqref{ballinflationeq1} is equal to
\begin{align*}
\avg{\Delta \in P_{\nu^{-b}}(\Delta')}&(\sum_{0 \leq k, l \leq \lceil\log_{2}(\nu^{-3b})\rceil}\sum_{\st{J \in \mc{F}_{1, k}\\J' \in \mc{F}_{2, l}}}\nms{\E_{J}g}_{L^{p/2}_{\#}(\ow_{\Delta})}^{2}\nms{\E_{J'}g}_{L^{p/2}_{\#}(\ow_{\Delta})}^{2})^{p/4}\nonumber\\
&\leq (\lceil \log_{2}(\nu^{-3b}) \rceil + 1)^{\frac{p}{2} - 2}\avg{\Delta \in P_{\nu^{-b}}(\Delta')}\sum_{0 \leq k, l \leq \lceil \log_{2}(\nu^{-3b})\rceil}F_{\Delta}(k, l).
\end{align*}
We then have
\begin{align}\label{ballinf3}
\begin{aligned}
\avg{\Delta \in P_{\nu^{-b}}(\Delta')}&\sum_{k, l = 0}^{\lceil \log_{2}(\nu^{-3b})\rceil}F_{\Delta}(k, l) =\\
& \avg{\Delta \in P_{\nu^{-b}}(\Delta')}F_{\Delta}(0, 0) + \sum_{l = 1}^{\lceil \log_{2}(\nu^{-3b})\rceil}\avg{\Delta \in P_{\nu^{-b}}(\Delta')}F_{\Delta}(0, l)\\
& + \sum_{k = 1}^{\lceil \log_{2}(\nu^{-3b})\rceil}\avg{\Delta \in P_{\nu^{-b}}(\Delta')}F_{\Delta}(k, 0) + \sum_{k, l = 1}^{\lceil \log_{2}(\nu^{-3b})\rceil}\avg{\Delta \in P_{\nu^{-b}}(\Delta')}F_{\Delta}(k, l).
\end{aligned}
\end{align}

We first consider the fourth sum on the right hand side of \eqref{ballinf3}. In this case, both families
of intervals satisfy \eqref{comp} in Lemma \ref{dyadic}. Thus applying Lemma \ref{dyadic} gives that
\begin{align*}
\sum_{k, l = 1}^{\lceil \log_{2}(\nu^{-3b})\rceil}&\avg{\Delta \in P_{\nu^{-b}}(\Delta')}F_{\Delta}(k, l)\\
& \lsm \lceil \log_{2}(\nu^{-3b})\rceil^{2}\exp(O(p))\nu^{-1}\geom(\sum_{J \in P_{(\nu^{b})}(I_i)}\nms{\E_{J}g}_{L^{p/2}_{\#}(\ow_{\Delta'})}^{2})^{p/2}.
\end{align*}
The first three terms on the right hand side of \eqref{ballinf3} are taken care of by trivial estimates so we omit their proof.
This completes the proof of the ball inflation lemma.
\end{proof}

\section{The iteration}\label{iteration}
In this section recall that $4 < p < 6$.
Let $\alpha$ be defined such that
$$\frac{1}{p/2} = \frac{1 - \alpha}{2} + \frac{\alpha}{p}.$$
Then $\alpha = \frac{p - 4}{p - 2}$ and $\nms{f}_{L^{p/2}} \leq \nms{f}_{L^{2}}^{1 - \alpha}\nms{f}_{L^{p}}^{\alpha}$.

The key estimate for our iteration is as follows which is just a combination of Lemma \ref{ballinflation}
and H\"{o}lder. This estimate is the analogue of \cite[Proposition 10.2]{sg} in our notation.

\begin{lemma}\label{mpup}
Suppose $\delta$ and $\nu$ were such that $\nu^{2b}\delta^{-1} \in \N$. Then
\begin{align*}
M_{p, b}(\delta, \nu) \lsm \nu^{-1/p}(\log\frac{1}{\nu^b})^{1/2}M_{p, 2b}(\delta, \nu)^{1 - \alpha}D_{p}(\frac{\delta}{\nu^b})^{\alpha}.
\end{align*}
\end{lemma}
\begin{proof}
H\"{o}lder and Lemma \ref{ballinflation} shows that
\begin{align}\label{mpupeq1}
&\avg{\Delta \in P_{\nu^{-b}}(B)}(\sum_{J \in P_{\nu^{b}}(I)}\nms{\E_{J}g}_{L^{2}_{\#}(\ow_\Delta)}^{2})^{\frac{p}{4}}(\sum_{J' \in P_{\nu^{b}}(I')}\nms{\E_{J'}g}_{L^{2}_{\#}(\ow_\Delta)}^{2})^{\frac{p}{4}}\nonumber\\
\leq&\avg{\Delta' \in P_{\nu^{-2b}}(B)}\avg{\Delta \in P_{\nu^{-b}}(\Delta')}(\sum_{J \in P_{\nu^{b}}(I)}\nms{\E_{J}g}_{L^{p/2}_{\#}(\ow_{\Delta})}^{2})^{p/4}(\sum_{J' \in P_{\nu^b}(I')}\nms{\E_{J'}g}_{L^{p/2}_{\#}(\ow_{\Delta})}^{2})^{p/4}\nonumber\\
\lsm &\nu^{-1}(\log\frac{1}{\nu^b})^{p/2}\avg{\Delta' \in P_{\nu^{-2b}}(B)}(\sum_{J \in P_{\nu^b}(I)}\nms{\E_{J}g}_{L^{p/2}_{\#}(\ow_{\Delta'})}^{2})^{p/4}(\sum_{J' \in P_{\nu^b}(I')}\nms{\E_{J'}g}_{L^{p/2}_{\#}(\ow_{\Delta'})}^{2})^{p/4}.
\end{align}
Applying the definition $\alpha$ and then H\"{o}lder controls the expression inside the average over $\Delta'$ by
\begin{align*}
(\sum_{J \in P_{\nu^b}(I)}&\nms{\E_{J}g}_{L^{2}_{\#}(\ow_{\Delta'})}^{2(1 - \alpha)}\nms{\E_{J}g}_{L^{p}_{\#}(\ow_{\Delta'})}^{2\alpha})^{p/4}(\sum_{J' \in P_{\nu^b}(I')}\nms{\E_{J'}g}_{L^{2}_{\#}(\ow_{\Delta'})}^{2(1 - \alpha)}\nms{\E_{J'}g}_{L^{p}_{\#}(\ow_{\Delta'})}^{2\alpha})^{p/4}\\
\leq &\bigg((\sum_{J \in P_{\nu^b}(I)}\nms{\E_{J}g}_{L^{2}_{\#}(\ow_{\Delta'})}^{2})^{p/4}(\sum_{J' \in P_{\nu^b}(I')}\nms{\E_{J'}g}_{L^{2}_{\#}(\ow_{\Delta'})}^{2})^{p/4}\bigg)^{1 - \alpha}\times\\
&\bigg((\sum_{J \in P_{\nu^b}(I)}\nms{\E_{J}g}_{L^{p}_{\#}(\ow_{\Delta'})}^{2})^{p/4}(\sum_{J' \in P_{\nu^b}(I')}\nms{\E_{J'}g}_{L^{p}_{\#}(\ow_{\Delta'})}^{2})^{p/4}\bigg)^{\alpha}.
\end{align*}
Combining this with \eqref{mpupeq1} and then H\"{o}lder in the average over $\Delta'$ shows that \eqref{mpupeq1} is
\begin{align*}
\lsm &\nu^{-1}(\log\frac{1}{\nu^b})^{p/2}\times\\
&\bigg(\avg{\Delta' \in P_{\nu^{-2b}}(B)}(\sum_{J \in P_{\nu^b}(I)}\nms{\E_{J}g}_{L^{2}_{\#}(\ow_{\Delta'})}^{2})^{p/4}(\sum_{J' \in P_{\nu^b}(I')}\nms{\E_{J'}g}_{L^{2}_{\#}(\ow_{\Delta'})}^{2})^{p/4}\bigg)^{1-\alpha}\times\\
&\bigg(\avg{\Delta' \in P_{\nu^{-2b}}(B)}(\sum_{J \in P_{\nu^b}(I)}\nms{\E_{J}g}_{L^{p}_{\#}(\ow_{\Delta'})}^{2})^{p/4}(\sum_{J' \in P_{\nu^b}(I')}\nms{\E_{J'}g}_{L^{p}_{\#}(\ow_{\Delta'})}^{2})^{p/4}\bigg)^{\alpha}.
\end{align*}
For the first average over $\Delta'$ we apply $l^2 L^2$ decoupling and the definition of the bilinear constant which contributes a
$M_{p, 2b}(\delta, \nu)^{p(1 - \alpha)}$. To control the second average over $\Delta'$ we apply the same arguments as in Lemma \ref{bil_lin} which contributes
a $D_{p}(\delta/\nu^b)^{p\alpha}$. Taking $1/p$ powers then completes the proof of this lemma.
\end{proof}

We now obtain an explicit upper bound for $D_{p}(\delta)$. We will assume already that for every $\vep > 0$
we have $D_{p}(\delta) \leq C_{\vep, p}\delta^{-\vep}$ for all $\delta \in \N^{-1}$ and some constant $C_{\vep, p}$
depending on $\vep$ and $p$. Such a bound can be proven by the same method as we show below (or alternatively we can rely on \cite{sg}, for example). By potentially
increasing $C_{\vep, p}$, we can assume that $C_{\vep, p} > 1$. We may also assume $\vep < 1/2$ since otherwise we can use the trivial bound.

Combining Lemma \ref{mpup} with Lemma \ref{reduction} then gives the following.
\begin{lemma}\label{recursion}
Suppose $N \geq 1$ and $\delta$ was such that $\delta^{1/2^{N + 1}} \in \N^{-1} \cap (0, 1/100)$. There is an absolute constant $C$ sufficiently large such that
\begin{align*}
D_{p}(\delta) \leq C^{N^2}(D_{p}(\delta^{1 - \frac{1}{2^{N + 1}}}) + \delta^{-\frac{1}{2^{N + 1}}(1 + \frac{2}{p}\sum_{j = 0}^{N}(\frac{2}{p - 2})^{j})}\prod_{j= 0}^{N}D_{p}(\delta^{1 - \frac{1}{2^{j + 1}}})^{\frac{p - 4}{p - 2}(\frac{2}{p - 2})^{N - j}}).
\end{align*}
\end{lemma}
\begin{proof}
Here $C$ will denote an absolute constant that will potentially change from line to line.
Applying Lemma \ref{mpup} repeatedly shows that if $\delta$ and $\nu$ were such that $\nu^{2^{N +1}}\delta^{-1} \in \N$, then
\begin{align*}
M_{p, 1}&(\delta, \nu)\\
&\leq (C\nu^{-1/p})^{\sum_{j = 0}^{N}(1 - \alpha)^j}\prod_{j = 0}^{N}(\log\frac{1}{\nu^{2^j}})^{\frac{1}{2}(1 - \alpha)^j}M_{p, 2^{N + 1}}(\delta, \nu)^{(1 - \alpha)^{N+1}}\prod_{j = 0}^{N}D_{p}(\frac{\delta}{\nu^{2^j}})^{\alpha(1 - \alpha)^j}\\
\end{align*}
for some large absolute constant $C$.
Thus if we choose $\nu = \delta^{1/2^{N + 1}}$, then after applying Lemma \ref{bil_lin} we have
\begin{align*}
M_{p, 1}&(\delta, \delta^{1/2^{N+1}})\\
& \leq C^{\sum_{j = 0}^{N + 1}(\frac{2}{p - 2})^j}\delta^{-\frac{1}{2^{N + 1}}\frac{1}{p}\sum_{j = 0}^{N}(\frac{2}{p - 2})^{j}}(\log\frac{1}{\delta})^{\frac{1}{2}\sum_{j = 0}^{N}(\frac{2}{p - 2})^{j}}\prod_{j = 0}^{N}D_{p}(\delta^{1 - \frac{1}{2^{j + 1}}})^{\frac{p - 4}{p - 2}(\frac{2}{p - 2})^{N - j}}.
\end{align*}
Combining this with Lemma \ref{reduction} shows that
\begin{align*}
D_{p}(\delta) \leq &C^{\sum_{j = 0}^{N + 1}(\frac{2}{p - 2})^{j}}(D_{p}(\delta^{1 - \frac{1}{2^{N + 1}}})\\
& + \delta^{-\frac{1}{2^{N + 1}}(1+\frac{1}{p}\sum_{j = 0}^{N}(\frac{2}{p - 2})^{j})}(\log\frac{1}{\delta})^{\frac{1}{2}\sum_{j = 0}^{N}(\frac{2}{p - 2})^{j}}\prod_{j = 0}^{N}D_{p}(\delta^{1 - \frac{1}{2^{j + 1}}})^{\frac{p - 4}{p - 2}(\frac{2}{p - 2})^{N - j}}).
\end{align*}
Recall that for all $a > 0$, $\log x \leq (a/e)x^{1/a}$ for $x > 0$. Thus $\log \frac{1}{\delta} \leq (p2^{N}/e)\delta^{-\frac{1}{p2^N}}$
and hence
\begin{align*}
D_{p}&(\delta)\\
 &\leq C^{N\sum_{j = 0}^{N + 1}(\frac{2}{p - 2})^{j}}(D_{p}(\delta^{1- \frac{1}{2^{N + 1}}}) + \delta^{-\frac{1}{2^{N + 1}}(1+\frac{2}{p}\sum_{j = 0}^{N}(\frac{2}{p - 2})^{j})}\prod_{j = 0}^{N}D_{p}(\delta^{1 - \frac{1}{2^{j + 1}}})^{\frac{p - 4}{p - 2}(\frac{2}{p - 2})^{N - j}})\\
&\leq C^{N^2}(D_{p}(\delta^{1- \frac{1}{2^{N + 1}}}) + \delta^{-\frac{1}{2^{N + 1}}(1+\frac{2}{p}\sum_{j = 0}^{N}(\frac{2}{p - 2})^{j})}\prod_{j = 0}^{N}D_{p}(\delta^{1 - \frac{1}{2^{j + 1}}})^{\frac{p - 4}{p - 2}(\frac{2}{p - 2})^{N - j}})
\end{align*}
which completes the proof of the lemma.
\end{proof}

Applying the monotonicity property in Remark \ref{equiv} then gives the following corollary.
\begin{cor}\label{rcor}
Suppose $N \geq 1$ and $\delta$ was such that $\delta^{1/2^{N + 1}} \in \N^{-1} \cap (0, 1/100)$. There is an absolute constant $C$ sufficiently large
such that
\begin{align}\label{rcoreq1}
D_{p}(\delta) \leq C^{N^2}(D_{p}(\delta^{1 - \frac{1}{2^{N + 1}}}) + \delta^{-\frac{N}{2^{N + 1}}}D_{p}(\delta)^{1 - (\frac{2}{p - 2})^{N + 1}}).
\end{align}
\end{cor}
\begin{proof}
From Remark \ref{equiv}, $D_{p}(\delta^{1 - \frac{1}{2^{j + 1}}}) \lsm D_{p}(\delta)$ where the implied constant is independent of $j, \delta$, and $p$
(the $p$ dependence is removed only because $2 \leq p \leq 6$, otherwise it is of the form $\exp(O(p))$).
\end{proof}

Fix $N \geq 1$.
Then there is some unspecified constant $C_{N, p}$ such that $$D_{p}(\delta) \leq (C_{N, p})\delta^{-\frac{N}{(4/(p - 2))^{N + 1}}}$$
for all $\delta \in \N^{-1}$.
For $n \geq 1$, let $\delta_n := 100^{-2^{N + 1}n}$. Then $\delta_{n}^{1/2^{N + 1}} \in \N^{-1} \cap (0, 100)$.
Apply Corollary \ref{rcor} to each $\delta_n$. We have two cases.\\

\noindent \textbf{Case 1:} The first term in \eqref{rcoreq1} dominates, that is,
\begin{align*}
D_{p}(\delta_n) \leq 2C^{N^2}D_{p}(\delta_{n}^{1 - \frac{1}{2^{N + 1}}}) \leq 2C^{N^{2}}(C_{N, p})\delta_{n}^{\frac{N}{(8/(p - 2))^{N + 1}}}\delta_{n}^{-\frac{N}{(4/(p - 2))^{N + 1}}}.
\end{align*}
If additionally that $\delta_{n} < (C_{N, p})^{-1}$, then this implies $$D_{p}(\delta_{n}) \leq 2C^{N^2}(C_{N, p})^{1 - \frac{N}{(8/(p - 2))^{N + 1}}}\delta_{n}^{-\frac{N}{(4/(p - 2))^{N + 1}}}.$$
Otherwise if additionally $\delta_n \geq (C_{N, p})^{-1}$, the trivial bound implies
$$D_{p}(\delta_n) \leq 2^{100/p} \delta_{n}^{-1/2} \leq 2^{25}(C_{N, p})^{1/2} \leq 2^{25}(C_{N, p})^{1 - \frac{N}{(8/(p - 2))^{N + 1}}}.$$
Therefore in this case $$D_{p}(\delta_{n}) \leq \max(2^{25}, 2C^{N^2})(C_{N, p})^{1 - \frac{N}{(8/(p - 2))^{N + 1}}}\delta_{n}^{-\frac{N}{(4/(p - 2))^{N + 1}}}.$$

\noindent \textbf{Case 2:} The second term in \eqref{rcoreq1} dominates, in which case we immediately have
$$D_{p}(\delta_n) \leq (2C^{N^2}\delta_{n}^{-N/2^{N + 1}})^{\frac{1}{(2/(p - 2))^{N + 1}}}.$$

Therefore in either case we have that there exists an absolute constant $C$ sufficiently large such that
\begin{align*}
D_{p}(\delta_n) \leq C^{\frac{N^2}{(2/(p - 2))^{N + 1}}}(C_{N, p})^{1 - \frac{N}{(8/(p - 2))^{N + 1}}}\delta_{n}^{-\frac{N}{(4/(p - 2))^{N + 1}}}.
\end{align*}

Now for $\delta \in \N^{-1}$, either $\delta \in (\delta_{n + 1}, \delta_n)$ for some $n \geq 1$ or $\delta \in (\delta_1, 1)$.
If $\delta \in (\delta_{n + 1}, \delta_n)$, then Corollary \ref{almostmult} shows that
\begin{align*}
D_{p}(\delta) &\leq C' D_{p}(\delta_n)D_{p}(\frac{\delta}{\delta_n}) \leq C'2^{25}100^{2^N} C^{\frac{N^2}{(2/(p - 2))^{N + 1}}}(C_{N, p})^{1 - \frac{N}{(8/(p - 2))^{N + 1}}}\delta^{-\frac{N}{(4/(p - 2))^{N + 1}}}
\end{align*}
for some absolute constants $C$ and $C'$. Note that in the second inequality we have used the trivial bound for $D(\delta/\delta_n)$.
If $\delta \in (\delta_1, 1)$, the trivial bound shows that
\begin{align*}
D_{p}(\delta) \leq 2^{25}\delta_{1}^{-1/2} = 2^{25}100^{2^N}.
\end{align*}
Thus for $\delta \in \N^{-1}$, there exists an absolute constant $C$ such that
\begin{align*}
D_{p}(\delta) \leq C^{2^N + \frac{N^2}{(2/(p - 2))^{N + 1}}} (C_{N, p})^{1 - \frac{N}{(8/(p - 2))^{N + 1}}}\delta^{-\frac{N}{(4/(p - 2))^{N + 1}}}.
\end{align*}

Let $P(C, \ld)$ be the statement that ``$D_{p}(\delta) \leq C\delta^{-\ld}$ for all $\delta \in \N^{-1}$". Then we have shown that
there is an absolute constant $C$ such that
\begin{align*}
P(C_{N, p}, \frac{N}{(4/(p - 2))^{N + 1}}) \implies P(C^{2^N + \frac{N^2}{(2/(p - 2))^{N + 1}}} (C_{N, p})^{1 - \frac{N}{(8/(p - 2))^{N + 1}}},\frac{N}{(4/(p - 2))^{N + 1}}).
\end{align*}
Iterating this infinitely many times shows that there is an absolute constant $C$ such that
\begin{align*}
P(C^{(\frac{8}{p - 2})^{N + 1}\frac{2^N}{N} + 4^{N + 1}N}, \frac{N}{(4/(p - 2))^{N + 1}})
\end{align*}
is true. This implies that for some absolute constant $C$ we have $P(C^{8^N}, \frac{N}{(4/(p - 2))^{N + 1}})$ is true.

Therefore we have shown that for $N \geq 100$, there is some absolute constant $C$ such that
$$D_{p}(\delta) \leq (C^{8^N}\delta^{-N/2^{N + 1}})^{\frac{1}{(2/(p - 2))^{N + 1}}}$$
for all $\delta \in \N^{-1}$.
We now choose $N$. We will be slightly inefficient and just choose $N$ such that
\begin{align*}
2^{-N} \leq (\log_{2}\delta^{-1})^{-1/4} \leq 2^{-N + 1}.
\end{align*}
Note that $N \geq 100$ if $\delta$ is sufficiently small. Then for $\delta$ sufficiently small,
\begin{align*}
C^{8^N}\delta^{-N/2^{N + 1}} \leq \exp(C'(\log\frac{1}{\delta})^{3/4}\log\log\frac{1}{\delta})
\end{align*}
for some absolute constant $C'$.
Finally, observe that
\begin{align*}
\frac{1}{(2/(p - 2))^{N + 1}} = \exp(-(N + 1)\log(\frac{2}{p - 2})) \leq C''(\log\frac{1}{\delta})^{\frac{1}{4}\log_{2}(\frac{p - 2}{2})}
\end{align*}
for some absolute constant $C''$.
Therefore for $\delta \lsm 1$, $\delta \in \N^{-1}$, and $4 < p < 6$, we have
\begin{align*}
D_{p}(\delta) \leq \exp(O((\log\frac{1}{\delta})^{\frac{3}{4} + \frac{1}{4}\log_{2}(\frac{p - 2}{2})}\log\log\frac{1}{\delta}))
\end{align*}
where all implied constants are absolute and independent of $p$. Having obtained the desired bound for all $\delta \lsm 1$,
we can obtain the remaining $\delta \in \N^{-1}$ by using the trivial bound.
This finishes the proof of Theorem \ref{main}.

\begin{rem}
We were a bit wasteful passing from Lemma \ref{recursion} to Corollary \ref{rcor}. For one fixed $p$, we can proceed as follows.
Let $N$ be chosen such that
\begin{align}\label{nchoice}
(\frac{4}{p - 2})^{N + 1}(\frac{p}{4} + \frac{p- 4}{4}\sum_{j = 1}^{N}(\frac{p - 2}{4})^{j}) \geq \frac{1}{\vep}(1 + \frac{2}{p}\sum_{j = 0}^{N}(\frac{2}{p - 2})^{j}).
\end{align}
Our choice of $N$ is somewhat opaque. To be more explicit, observe that if $p = 4$, \eqref{nchoice} is the statement that
$$2^{N + 1} \geq \frac{1}{\vep}(1 + \frac{N + 1}{2}).$$
If $p = 5$, \eqref{nchoice} is the statement that
$$(\frac{4}{3})^{N + 1}(\frac{5}{4} + \frac{1}{4}\sum_{j = 1}^{N}(\frac{3}{4})^j) \geq \frac{1}{\vep}(1 + \frac{2}{5}\sum_{j = 0}^{N}(\frac{2}{3})^j).$$
Finally if $p = 6$, \eqref{nchoice} is the statement that
$$\frac{3}{2} + \frac{1}{2}N \geq \frac{1}{\vep}(1 + \frac{1}{3}\sum_{j = 0}^{N}\frac{1}{2^j}).$$
If $p \in (4.0001, 6)$ is fixed, then one can show from Lemma \ref{recursion} that for $\delta$ sufficiently small (depending on $p$),
$$D_{p}(\delta) \leq \exp(C_{p}(\log\frac{1}{\delta})^{1 - \frac{1}{2 + \log_{4/(p - 2)}4}}).$$
Here $C_p$ tends to infinity as $p$ increases to 6.
This exponent is better than $\frac{3}{4} + \frac{1}{4}\log_{2}(\frac{p - 2}{2})$ for all $p \in (4, 6)$ but since in our application we want to avoid dependence on $p$
inside the exponential, we do not use this estimate in the proof of Theorem \ref{main2}.
\end{rem}

\appendix
\section{The case of the circle\\(by Jean Bourgain and Zane Kun Li)}\label{appendix}
The square root cancellation estimate \eqref{targetexpsum} is a consequence of explicit decoupling for the circle which we derive below
from our quantitative bounds for decoupling for the parabola and the Pramanik-Seeger iteration.

\subsection{Explicit decoupling for $(t, at^2)$}
Previously we obtained an explicit upper bound for the decoupling constant for the curve $(t, t^2)$.
We now show how this gives information about the decoupling constant for the curve $(t, at^2)$ where $a \in [1/100, 100]$.

Let $\mc{P}_{a}(t) := at^2$ and given a function $h$, define
\begin{align*}
\ta_{J, C\delta^2}(h) := \{(s, h(s) + t): s \in J, |t| \leq C\delta^2\}.
\end{align*}
For a function $h$ and $\delta \in \N^{-1}$, we let $D_{p}(\delta, C, h)$ be the best constant such that
\begin{align*}
\nms{f}_{p} \leq D_{p}(\delta, C, h)(\sum_{J \in P_{\delta}([0, 1])}\nms{f_{\ta_{J, C\delta^2}(h)}}_{p}^{2})^{1/2}
\end{align*}
for all $f$ with Fourier support in $\ta_{[0, 1], C\delta^2}(h)$.
\begin{lemma}\label{cpar}
If $C \in [1/100^2, 100^2]$ and $4 < p < 6$, then for all $\delta \in \N^{-1}$ sufficiently small,
$$D_{p}(\delta, C, \mc{P}_1) \lsm \exp(O((\log \frac{1}{\delta})^{1 - \sigma_{p}}(\log\log\frac{1}{\delta})^{O(1)}))$$
where $\sigma_p := \frac{1}{4}(1 - \log_{2}(\frac{p - 2}{2}))$
and all implied constants are absolute.
\end{lemma}
\begin{proof}
Using Remark 5.2 of \cite{bd} (or alternatively Section 4.1 of \cite{thesis}) to convert to the extension operator formulation of decoupling and then using that
$w_{CB(0, R)} \sim_{C} w_{B(0, R)}$ shows that $$D_{p}(\delta, C, \mc{P}_1) \sim_{C} D_{p}(\delta, 1, \mc{P}_1).$$
Remark 5.2 of \cite{bd} and Theorem \ref{main} show that for $\delta$ sufficiently small and $4 < p < 6$,
\begin{align*}
D_{p}(\delta, 1, \mc{P}_1) \lsm \exp(O((\log \frac{1}{\delta})^{1 - \sigma_{p}}(\log\log\frac{1}{\delta})^{O(1)}))
\end{align*}
where all implied constants are absolute.
This completes the proof of Lemma \ref{cpar}.
\end{proof}

\begin{lemma}\label{apar}
We have $$D_{p}(\delta, C, \mc{P}_a) = D_{p}(\delta, C/a, \mc{P}_1).$$
\end{lemma}
\begin{proof}
It suffices to show that for any $r > 0$ that
$$D_{p}(\delta, C, \mc{P}_a) \leq D_{p}(\delta, Cr, P_{ar}).$$
Suppose $f$ has Fourier support in $\ta_{[0, 1], C\delta^2}(\mc{P}_a)$. A change of variables
gives $$\nms{f}_p = r^{1/p - 1}\nms{(f_r)_{\ta_{[0, 1], Cr\delta^2}(\mc{P}_{ra})}}_{p}$$
where $f_{r}(x_1, x_2) = rf(x_1, rx_2)$ since $$f(x) = \frac{1}{r}(f_r)_{\ta_{[0, 1], Cr\delta^2}(\mc{P}_{ra})}(x_1, \frac{x_2}{r}).$$
Now apply the definition of $D_{p}(\delta, Cr, \mc{P}_{ar})$ and undo the change of variables.
This completes the proof of Lemma \ref{apar}.
\end{proof}

Combining Lemmas \ref{cpar} and \ref{apar} gives the following corollary
which gives the decoupling constant for the curve $(t, at^2)$ where $a$ lies in a bounded interval
away from 0.
\begin{cor}\label{acor}
For $a, C \in [1/100, 100]$, $\delta \in \N^{-1}$ sufficiently small, and $4 < p < 6$,
\begin{align*}
D_{p}(\delta, C, \mc{P}_a) \leq \exp(O((\log \frac{1}{\delta})^{1 - \sigma_{p}}(\log\log\frac{1}{\delta})^{O(1)}))
\end{align*}
where $\sigma_p$ is as in Lemma \ref{cpar} and the implied constants are absolute.
\end{cor}

\subsection{Explicit decoupling for $(t, h(t))$}\label{tht}
\subsubsection{Setting up the iteration}
For $r \in \N^{-1}$, let $$\mathfrak{D}_{p}(r, a) := D_{p}(r, 2, \mc{P}_a)$$ and  $$\mf{D}_{p}(r) := \sup_{1/4 \leq a \leq 1}\mf{D}_{p}(r, a).$$
Given $\tau^{1/2} \in \N^{-1}$, observe that by rescaling the interval $[0, \tau]$ to $[0, 1]$,
\begin{align*}
\nms{f_{\ta_{[0, \tau], 2\tau^3}(\mc{P}_a)}}_{p} \leq \mf{D}_{p}(\tau^{1/2}, a)(\sum_{J \in P_{\tau^{3/2}}([0, \tau])}\nms{f_{\ta_{J, 2\tau^3}(\mc{P}_a)}}_{p}^{2})^{1/2}
\end{align*}
for all $f$ with Fourier support in $\ta_{[0, \tau], 2\tau^3}(\mc{P}_a)$.

\begin{lemma}\label{approx}
Let $1/4 \leq a \leq 1$ and $\tau^{1/2} \in \N^{-1}$.
For all functions $h$ satisfying $|h(s) - \mc{P}_{a}(s)| \leq \tau^{3}$ for $s \in [0, \tau]$, we have
\begin{align*}
\nms{f_{\ta_{[0, \tau], \tau^3}(h)}}_{p} \leq \mf{D}_{p}(\tau^{1/2}, a)(\sum_{J \in P_{\tau^{3/2}}([0, \tau])}\nms{f_{\ta_{J, \tau^3}(h)}}_{p}^{2})^{1/2}
\end{align*}
for all $f$ with Fourier support in $\ta_{[0, \tau], \tau^3}(h)$.
\end{lemma}
\begin{proof}
It suffices to show that $\ta_{[0, \tau], \tau^3}(h) \subset \ta_{[0, \tau], 2\tau^{3}}(\mc{P}_{a})$.
This is equivalent to showing that $(s, h(s) + t) \in \ta_{[0, \tau], 2\tau^3}(\mc{P}_a)$ for $s \in [0, \tau]$
and $|t| \leq \tau^3$. It suffices to show that for each $s \in [0, \tau]$ we have $|h(s) + t - \mc{P}_{a}(s)| \leq 2\tau^3$ for all $|t| \leq \tau^3$. But this is immediate
from our hypothesis. This completes the proof of Lemma \ref{approx}.
\end{proof}

\begin{lemma}\label{shift}
Let $\tau^{1/2} \in \N^{-1}$ and $[\ell, \ell + \tau] \subset [0, 1]$. Suppose $h \in C^{3}([\ell, \ell + \tau])$ is such that
$1/2 \leq h''(s) \leq 2$ and $|h'''(s)| \leq 2$ for $s \in [\ell, \ell + \tau]$. Then
\begin{align*}
\nms{f_{\ta_{[\ell, \ell + \tau], \tau^3}(h)}}_{p} \leq \mf{D}_{p}(\tau^{1/2})(\sum_{J \in P_{\tau^{3/2}}([\ell, \ell + \tau])}\nms{f_{\ta_{J, \tau^3}(h)}}_{p}^{2})^{1/2}
\end{align*}
for all $f$ with Fourier support in $\ta_{[\ell, \ell + \tau], \tau^3}(h)$.
\end{lemma}
\begin{proof}
Fix arbitrary $h$ satisfying the hypotheses and let $f$ have Fourier support in $\ta_{[\ell, \ell + \tau], \tau^3}(h)$.
Define $T = (\begin{smallmatrix} 1 & 0 \\ -h'(\ell) & 1\end{smallmatrix})$. Then
\begin{align*}
T(\ta_{[\ell, \ell + \tau], \tau^3}(h) - (\ell, h(\ell))) = \ta_{[0, \tau], \tau^3}(\wt{h})
\end{align*}
where $$\wt{h}(s) = h(\ell + s) - h(\ell) - h'(\ell)s.$$
Note that $\wt{h}(0) = \wt{h}'(0) = 0$ and hence by Taylor's theorem (and that $\wt{h} \in C^{3}([0, \tau])$),
\begin{align*}
|\wt{h}(s) - \mc{P}_{\wt{h}''(0)/2}(s)| \leq \tau^3/3.
\end{align*}

We have
\begin{align*}
|f(x)| &= |\int_{\ta_{[\ell, \ell + \tau], \tau^3}(h)}\wh{f}(\xi)e(x \cdot \xi)\, d\xi|\\
&= |\int_{\ta_{[0, \tau], \tau^3}(\wt{h})}\wh{f}(T^{-1}\eta + (\ell, h(\ell)))e(\eta \cdot T^{-t}x)\, d\eta|.
\end{align*}
Let $g(y) := f(T^{t}y)e(T(\ell, h(\ell))  \cdot y)$. Then
$\wh{f}(T^{-1}\eta + (\ell, h(\ell))) = \wh{g}(\eta)$ and $g$ has Fourier support in $\ta_{[0,\tau], \tau^3}(\wt{h})$.
Applying Lemma \ref{approx}, we then have
\begin{align*}
\nms{f}_{p} = \nms{g}_{p} \leq \mf{D}_{p}(\tau, \frac{\wt{h}''(0)}{2})(\sum_{J \in P_{\tau^{3/2}}([0, \tau])}\nms{g_{\ta_{J, \tau^3}(\wt{h})}}_{p}^{2})^{1/2}.
\end{align*}
The same changes of variables as above shows that as $J$ runs through $P_{\tau^{3/2}}([0, \tau])$ the expression $\nms{g_{\ta_{J, \tau^{3}}(\wt{h})}}_{p}$
is equal to $\nms{f_{\ta_{J', \tau^3}(h)}}_{p}$ as $J'$ runs through $P_{\tau^{3/2}}([\ell, \ell + \tau])$.
Finally, observing that $\wt{h}''(0)/2 = h''(\ell)/2$  and using that $1/2 \leq h''(s) \leq 2$ for $s \in [\ell, \ell + \tau]$
completes the proof of Lemma \ref{shift}.
\end{proof}

\subsubsection{Selection of parameters}\label{choice}
Throughout the remainder of this appendix, fix $\delta \in (0, 1)$ sufficiently small. Note that unlike previously, we impose no integrality
conditions whatsoever on $\delta$.
That is, in this appendix, we do not assume that $\delta \in \N^{-1}$.
Let $C_0 < 1/100$ be a sufficiently small absolute constant.
Choose $N \in \N$ such that
\begin{align}\label{ndeltchoice}
C_{0}^{3 \cdot 3^{N}}\leq \delta \leq C_{0}^{2 \cdot 3^{N}}.
\end{align}
Next choose $\tau_{0} := C_{0}^{2\cdot 2^{N}}.$
For simplicity of notation, let $$\tau_{j} := \tau_{0}^{(3/2)^{j}}$$
and observe that $\tau_{j}^{1/2} \in \N^{-1}$ for $j = 0, 1, \ldots, N$.
From our choice of $\tau_0$, $C_{0}^{3 \cdot 3^N} = \tau_{0}^{(\frac{3}{2})^{N + 1}}$
and $C_{0}^{2 \cdot 3^N} = \tau_{0}^{(\frac{3}{2})^{N}}$ and hence
$$\tau_{N + 1} = \tau_{0}^{(\frac{3}{2})^{N + 1}} \leq \delta \leq \tau_{0}^{(\frac{3}{2})^{N}} = \tau_N.$$
Our choice of $N$ in \eqref{ndeltchoice} gives that $N \lsm \log\log\frac{1}{\delta}$,
\begin{align*}
\tau_{0}^{-1} = \exp(O((\log \frac{1}{\delta})^{\log_{3}2})),
\end{align*}
and $\delta^{-1}\lceil \delta^{-2}\rceil \tau_{0}^{-2} = O_{\vep}(\delta^{-3 - \vep})$.

\subsubsection{The iteration}
Let $\mc{C}$ be the class of functions in $C^{3}([0, 1])$ such that
\begin{itemize}
\item $h(0) = h'(0) = h'''(0) = 0$
\item $1/2 \leq h''(s) \leq 2$ for $s \in [0, 1]$.
\end{itemize}
Note that for $h \in \mc{C}$, from the fundamental theorem of calculus, $|h'''(s)| \leq 2$
for all $s \in [0, 1]$.
We now apply an iterative argument to upgrade the information in Corollary \ref{acor} to give us
information about $D_{p}(\delta, 1, h)$ for $h \in \mc{C}$.
For $r \in \N^{-1}$, we define
\begin{align*}
D_{p}(r, \mc{C}) := \sup_{h \in \mc{C}}D_{p}(r, 1, h).
\end{align*}

\begin{lemma}\label{dbd}
With our choice of $\tau_0$ and $N$ as in Section \ref{choice},
\begin{align*}
D_{p}(\tau_{N + 1}, \mc{C}) \leq \tau_{0}^{-1/2}\prod_{j = 0}^{N}\mf{D}_{p}(\tau_{j}^{1/2}).
\end{align*}
\end{lemma}
\begin{proof}
First we decouple trivially to scale $\tau_0$ via the triangle inequality. Let $f$ be
a function with Fourier support in $\ta_{[0, 1], \tau_{0}^{2}}(h)$. Then since $\tau_0 \in \N^{-1}$, the triangle inequality gives
\begin{align}\label{step1}
\nms{f_{\ta_{[0, 1], \tau_{0}^{2}}(h)}}_{p} \leq \tau_{0}^{-1/2}(\sum_{J \in P_{\tau_{0}}([0, 1])}\nms{f_{\ta_{J, \tau_{0}^{2}}(h)}}_{p}^{2})^{1/2}.
\end{align}

By our choice of $\tau_0$, $\tau_{0}^{1/2} \in \N^{-1}$. Our assumptions on $h$ combined with Lemma \ref{shift} show that for each $J \in P_{\tau_0}([0, 1])$,
\begin{align}\label{step2}
\nms{f_{\ta_{J, \tau_{0}^{3}}(h)}}_{p} \leq \mf{D}_{p}(\tau_{0}^{1/2})(\sum_{J' \in P_{\tau_{0}^{3/2}}(J)}\nms{f_{\ta_{J', \tau_{0}^{3}}(h)}}_{p}^{2})^{1/2}
\end{align}
for all $f$ with Fourier support in $\ta_{J, \tau_{0}^{3}}(h)$.
Combining \eqref{step1} with \eqref{step2} we have shown that if $f$ is a function with Fourier support in $\ta_{[0, 1], \tau_{0}^{3}}(h)$,
then
\begin{align*}
\nms{f_{\ta_{[0, 1], \tau_{0}^{3}}(h)}}_{p} \leq \tau_{0}^{-1/2}\mf{D}_{p}(\tau_{0}^{1/2})(\sum_{J \in P_{\tau_{0}^{3/2}}([0, 1])}\nms{f_{\ta_{J, \tau_{0}^{3}}(h)}}_{p}^{2})^{1/2}.
\end{align*}

By our choice of $\tau_0$, $(\tau_{0}^{3/2})^{1/2} = \tau_{1}^{1/2} \in \N^{-1}$. Lemma \ref{shift} then shows that for each $J \in P_{\tau_{0}^{3/2}}([0, 1])$,
\begin{align*}
\nms{f_{\ta_{J, (\tau_{0}^{3/2})^{3}}(h)}}_{p} \leq \mf{D}_{p}((\tau_{0}^{3/2})^{1/2})(\sum_{J' \in P_{\tau_{0}^{(3/2)^2}}(J)}\nms{f_{\ta_{J', (\tau_{0}^{3/2})^{3}}(h)}}_{p}^{2})^{1/2}
\end{align*}
for all $f$ with Fourier support in $\ta_{J, (\tau_{0}^{3/2})^{3}}(h)$. This shows that if $f$ is a function with Fourier support in
$\ta_{[0, 1], (\tau_{0}^{3/2})^{3}}(h)$, then
\begin{align*}
\nms{f_{\ta_{[0, 1], (\tau_{0}^{3/2})^{3}}(h)}}_{p} \leq \tau_{0}^{-1/2}\mf{D}_{p}(\tau_{0}^{1/2})\mf{D}_{p}((\tau_{0}^{3/2})^{1/2})(\sum_{J \in P_{\tau_{0}^{(3/2)^2}}([0, 1])}\nms{f_{\ta_{J, (\tau_{0}^{3/2})^{3}}(h)}}_{p}^{2})^{1/2}.
\end{align*}

By our choice of $\tau_{0}$, $\tau_{0}^{\frac{1}{2} \cdot (\frac{3}{2})^{j}} = \tau_{j}^{1/2} \in \N^{-1}$ for $j = 0, \ldots, N$.
Applying the above process a total of $N + 1$ times shows that if $f$ is a function with Fourier support in $\ta_{[0, 1], \tau_{0}^{3(3/2)^N}}(h)$, then
\begin{align*}
\nms{f_{\ta_{[0, 1], \tau_{0}^{3(3/2)^N}}(h)}}_{p} \leq \tau_{0}^{-1/2}\prod_{j = 0}^{N}\mf{D}_{p}(\tau_{0}^{\frac{1}{2} \cdot (\frac{3}{2})^{j}})(\sum_{J \in P_{\tau_{0}^{(3/2)^{N + 1}}}([0, 1])}\nms{f_{\ta_{J, \tau_{0}^{3(3/2)^N}}(h)}}_{p}^{2})^{1/2}.
\end{align*}
Applying the definitions of $D_{p}(\tau_{0}^{(3/2)^{N + 1}}, \mc{C})$ and $\tau_j$, then completes the proof of Lemma \ref{dbd}.
\end{proof}

Combining this lemma with Corollary \ref{acor} immediately gives the following.
\begin{cor}\label{dbdcor}
With our choice of $\delta, \tau_0$, and $N$ as in Section \ref{choice} and $4 < p < 6$, we have
$$D_{p}(\tau_{N + 1}, \mc{C}) \leq \exp(O((\log\frac{1}{\delta})^{1 - \sigma_p}(\log\log\frac{1}{\delta})^{O(1)}))$$
where the implied constant is absolute and $\sigma_p$ is defined as in Lemma \ref{cpar}.
\end{cor}

\subsection{The exponential sum estimate}
We now prove \eqref{targetexpsum}.
Observe that
\begin{align*}
H(\xi) := \frac{1 - \sqrt{1 - \xi^2\tau_{0}^{2}}}{\tau_{0}^{2}}
\end{align*}
is such that $H(0) = H'(0) = H'''(0) = 0$ and for $\tau_0$ sufficiently small (which can be satisfied if $C_0$ is sufficiently small) $$\frac{1}{2} \leq H''(\xi) = \frac{1}{(1 - \xi^{2}\tau_{0}^{2})^{3/2}} \leq 2$$ for $\xi \in [0, 1]$.
Therefore since $\tau_{0}^{3(3/2)^{N}} = \tau_{N+ 1}^{2}$ and $H \in \mc{C}$, Corollary \ref{dbdcor} implies the following lemma.
\begin{lemma}\label{decoupling2}
For an interval $J$, define the extension operator
\begin{align*}
(\E'_{J}g)(x) := \int_{J}g(\xi)e(\xi x_1 + \frac{1 - \sqrt{1 - \xi^2\tau_{0}^{2}}}{\tau_{0}^{2}} x_2)\, d\xi.
\end{align*}
Then for $4 < p < 6$, we have
\begin{align*}
\nms{\E'_{[0, 1]}g}_{L^{p}(B)} \leq \exp(O((\log \frac{1}{\delta})^{1 - \sigma_p}(\log\log\frac{1}{\delta})^{O(1)}))(\sum_{J \in P_{\tau_{N + 1}}([0, 1])}\nms{\E'_{J}g}_{L^{p}(w_B)}^{2})^{1/2}
\end{align*}
for all $g: [0, 1] \rightarrow \C$ and all squares $B$ of side length $\tau_{N + 1}^{-2}$.
\end{lemma}

Rescaling $[0, 1]$ to $[0, \tau_0]$ and using parallel decoupling, Lemma \ref{decoupling2} implies the following.
\begin{lemma}\label{decoupling}
For an interval $J$, define the extension operator
$$(\E_{J}g)(x) := \int_{J}g(\xi)e(\xi x_1 + (1 - \sqrt{1 - \xi^2})x_2)\, d\xi.$$
Then for $4 < p < 6$, we have
\begin{align*}
\nms{\E_{[0, \tau_0]}g}_{L^{p}(B)} \leq \exp(O((\log\frac{1}{\delta})^{1 - \sigma_p}(\log\log\frac{1}{\delta})^{O(1)}))(\sum_{J \in P_{\tau_{N + 1}\tau_0}([0, \tau_0])}\nms{\E_{J}g}_{L^{p}(w_B)}^{2})^{1/2}
\end{align*}
for all $g: [0, \tau_0] \rightarrow \C$ and all squares $B$ of side length $(\tau_{N + 1}\tau_0)^{-2}$.
\end{lemma}
Observe that $\tau_{N + 1}^{-2} = \tau_{N}^{-3} \leq \delta^{-3}$.
Parallel decoupling allows us to increase the side length of $B$ in the above lemma to $\delta^{-1}\lceil\delta^{-2}\rceil\tau_{0}^{-2}= O_{\vep}(\delta^{-3 - \vep})$.
Note that $\lceil \delta^{-2} \rceil \tau_{0}^{-2}$ is an integer.

Having obtained a decoupling theorem, our next goal is to obtain the following exponential sum estimate which is an instance of square root cancellation.
\begin{prop}\label{srtc}
Let $A$ be a set of points $(x, y)$ with $x^{2} + y^{2} = 1$ and separated by distance at least $\delta$.
Then for $4 < p < 6$,
\begin{align*}
\nms{\sum_{a \in A}e^{2\pi i a \cdot z}}_{L^{p}_{\#}(B)} \leq \exp(O((\log\frac{1}{\delta})^{1 - \sigma_p}(\log\log\frac{1}{\delta})^{O(1)}))|A|^{1/2}
\end{align*}
for all squares $B$ of side length $\delta^{-1}\lceil\delta^{-2}\rceil\tau_{0}^{-2}$.
\end{prop}
Since $\tau_0 \in \N^{-1}$, we can partition the unit circle into $\tau_{0}^{-1}$ many arcs $\{C_{\tau_0}\}$ of arc length $\tau_{0}^{-1}$.
From our choice of $\tau_0$ and the triangle inequality, it suffices to instead
prove the following result.
\begin{lemma}\label{red1}
Let $A$ be a set of points $(x, y)$ with $x^{2} + y^{2} = 1$ and separated by distance at least $\delta$.
Then for $4 < p < 6$ and each arc $C_{\tau_0}$
of arc length $\tau_0$, we have
\begin{align*}
\nms{\sum_{a \in A \cap C_{\tau_0}}e^{2\pi i a \cdot z}}_{L^{p}_{\#}(B)} \leq \exp(O((\log\frac{1}{\delta})^{1 - \sigma_p}(\log\log\frac{1}{\delta})^{O(1)}))|A \cap C_{\tau_0}|^{1/2}
\end{align*}
for all squares $B$ of side length $\delta^{-1}\lceil\delta^{-2}\rceil\tau_{0}^{-2}$.
\end{lemma}

Without loss of generality we may assume that $C_{\tau_0}$ is the circular arc of arc length $\tau_0$ with left endpoint at the origin.
Such an arch is given by $\{(\sin\ta, 1 - \cos\ta): \ta \in [0, \tau_0]\}$.
Recall
that $A$ is a $\delta$-separated set and $\sin\tau_0 \leq \tau_0$. Observe that if $\tau_0$ was sufficiently small (and so $\delta$ is smaller than
a sufficiently small absolute constant), a subarc of $C_{\tau_0}$ that lies above an interval of length $\tau_{N + 1}\tau_0 \leq \delta\tau_0$
has at most one point in $A$. Indeed, if $a \leq \tau_{0}(1 - \tau_{N + 1})$ and $\ell = \tau_{N + 1}\tau_0$, the distance
between the points $(a, 1 - \sqrt{1 - a^2})$ and $(a + \ell, 1 - \sqrt{1 - (a + \ell)^2})$ is
\begin{align*}
\ell(1 + (\frac{2a + \ell}{\sqrt{1 - a^2} + \sqrt{1 - (a + \ell)^2}})^{2})^{1/2} \leq \delta\frac{2\tau_0}{\sqrt{1 - \tau_{0}^{2}}} < \delta
\end{align*}
if $\tau_0$ is sufficiently small.

Thus Lemma \ref{red1} follows from Lemma \ref{decoupling}
if we pick $g$ to be a smooth approximation of $\sum_{a \in A \cap C_{\tau_0}}1_{\xi = a}$.

Having proven Proposition \ref{srtc} for all $4 < p < 6$ and the implied constant there is independent of $p$, we can use H\"{o}lder to obtain an $L^6$ estimate
assuming an extra condition about $|A|$.
\begin{lemma}\label{mainlemma}
Let $A$ be a set of points $(x, y)$ with $x^2 + y^2 = 1$ and separated by a distance at least $\delta$.
If $|A| > \exp((\log\frac{1}{\delta})^{1 - o(1)})$, then
\begin{align*}
\nms{\sum_{a \in A}e^{2\pi i a \cdot z}}_{L^{6}_{\#}([0,\delta^{-1}\lceil\delta^{-2}\rceil\tau_{0}^{-2}]^{2})} \ll |A|^{1/2 + o(1)}.
\end{align*}
\end{lemma}
\begin{proof}
Applying H\"{o}lder and Proposition \ref{srtc} shows that
\begin{align*}
\nms{\sum_{a \in A}e^{2\pi i a \cdot z}}_{L^{6}_{\#}([0,\delta^{-1}\lceil\delta^{-2}\rceil\tau_{0}^{-2}]^{2})} &\leq \nms{\sum_{a \in A}e^{2\pi i a \cdot z}}_{L^{p}_{\#}([0, \delta^{-1}\lceil\delta^{-2}\rceil\tau_{0}^{-2}]^{2})}^{p/6}|A|^{1 - \frac{p}{6}}\\
&\leq \exp(O((\log\frac{1}{\delta})^{1 - \sigma_{p}}(\log\log\frac{1}{\delta})^{O(1)}))|A|^{\frac{1}{2}(1- \frac{p}{6})}|A|^{1/2}.
\end{align*}
Therefore
\begin{align*}
\nms{\sum_{a \in A}e^{2\pi i a \cdot z}}_{L^{6}_{\#}([0, \delta^{-1}\lceil\delta^{-2}\rceil\tau_{0}^{-2})} \leq \min_{\tau > 0}(\exp(c(\log\frac{1}{\delta})^{1 - \tau})|A|^{C\tau})|A|^{1/2}
\end{align*}
for some absolute constants $C$ and $c$. The conclusion of Lemma \ref{mainlemma} then follows from the hypothesis that $\log|A| > (\log\frac{1}{\delta})^{1 - o(1)}$.
\end{proof}
%

\subsection{Proof of Theorem \ref{main2}}
Using Lemma \ref{mainlemma} and periodicity we now prove Theorem \ref{main2}.

Fix an arbitrary sufficiently large integer $R$. The set
of $(z_1, \ldots, z_6) \in (\Z + i\Z)^{6}$ such that $z\ov{z} = R$ and $z_1 + z_2 + z_3 = z_4 + z_5 + z_6$
is equal to
\begin{align*}
\int_{[0, 1]^{2}}&|\sum_{\st{|n| \leq \sqrt{R},  s \in \{\pm 1\}\\\sqrt{R - n^2} \in \N \cup \{0\}}}e^{2\pi i (nx + (\sqrt{R} + s \sqrt{R - n^{2}})y)}|^{6}\, dx\, dy\\
& = \int_{[0, 1]^{2}}|\sum_{\st{|n| \leq \sqrt{R},  s \in \{\pm 1\}\\\sqrt{R - n^2} \in \N \cup \{0\}}}e^{2\pi i (\frac{n}{\sqrt{R}}(\sqrt{R}x) + (1 + s \sqrt{1 - (\frac{n}{\sqrt{R}})^{2}})(\sqrt{R}y))}|^{6}\, dx\, dy\\
&= \frac{1}{R}\int_{[0, \sqrt{R}]^2}|\sum_{\st{|n| \leq \sqrt{R},  s \in \{\pm 1\}\\\sqrt{R - n^2} \in \N \cup \{0\}}}e^{2\pi i (\frac{n}{\sqrt{R}}x + (1 + s \sqrt{1 - (\frac{n}{\sqrt{R}})^{2}})y)}|^{6}\, dx\, dy.
\end{align*}
Taking conjugates and using the triangle inequality, it suffices to study
\begin{align*}
\frac{1}{R}\int_{[0, \sqrt{R}]^2}|\sum_{\st{|n| \leq \sqrt{R}\\\sqrt{R - n^2} \in \N \cup \{0\}}}e^{2\pi i (\frac{n}{\sqrt{R}}x + (1- \sqrt{1 - (\frac{n}{\sqrt{R}})^{2}})y)}|^{6}\, dx\, dy.
\end{align*}
Since the integrand is periodic with period $\sqrt{R}$ in both $x$ and $y$ variables, using periodicity the above is in fact equal to
\begin{align*}
\frac{1}{X^2}\int_{[0, X]^2}|\sum_{\st{|n| \leq \sqrt{R}\\\sqrt{R - n^2} \in \N \cup \{0\}}}e^{2\pi i (\frac{n}{\sqrt{R}}x + (1- \sqrt{1 - (\frac{n}{\sqrt{R}})^{2}})y)}|^{6}\, dx\, dy
\end{align*}
where $X = c\sqrt{R}$ for some integer $c$.
Applying Lemma \ref{mainlemma} with $\delta = 1/\sqrt{R}$ and $c = \lceil\delta^{-2}\rceil\tau_{0}^{-2}$ (where $\tau_0$ relates to $\delta$ as in Section \ref{choice})
then shows that the above integral is $\ll |A|^{3 + o(1)}$.
This completes the proof of Theorem \ref{main2}.

\bibliographystyle{amsplain}
\bibliography{2ddec_v5}

\end{document}